\documentclass[12pt]{article}
\usepackage{fancyhdr,graphicx}
\usepackage{subfigure} 
\usepackage{epstopdf}
\usepackage{float}
\usepackage{verbatim}
\usepackage{float} 
\usepackage{xcolor}
\usepackage{amssymb}
\usepackage{amsthm}
\usepackage{amsmath}
\usepackage{amsfonts}
\usepackage{bbm}
\usepackage{CJK}

\usepackage{mathrsfs}
\usepackage{algorithm}
\usepackage{algorithmic}
\usepackage{latexsym,bm}

\newtheorem{claim}{Claim}
\newtheorem{observation}{Observation}[section]
\newtheorem{theorem}{Theorem}[section]
\newtheorem{lemma}{Lemma}[section]

\newtheorem{example}{Example}[section]
\newtheorem{corollary}{Corollary}[section]

{
    
}
\vspace{-0.6cm}
\graphicspath{{figure/}}                                 
\usepackage[a4paper]{geometry}                           
\begin{document}

\title{\bf The complexity of total edge domination and some related results on trees}
\date{}

  \author{
\rmfamily Zhuo Pan$^{\ast}$, Yu Yang$^{\ast}$,  Xianyue Li, Shou-Jun Xu$^{\dagger}$\\
    {\rmfamily \small School of Mathematics and Statistics, Gansu Key Laboratory of Applied Mathematics }\\
    {\rmfamily \small and Complex Systems, Lanzhou University, Lanzhou, Gansu 730000, China}
     }

\footnotetext[1]{The first two authors contributed equally to this paper.}
\footnotetext[2]{Corresponding author. E-mail address: shjxu@lzu.edu.cn (S.-J. Xu)}

      \maketitle


\noindent\rmfamily\small{\bf Abstract:}
For a graph $G = (V, E)$ with vertex set $V$ and edge set $E$, a subset $F$ of
$E$ is called an $\emph{edge dominating set}$ (resp. a $\emph{total edge dominating set}$) if every edge in $E\backslash F$ (resp. in $E$) is adjacent to at least one edge in $F$, the minimum cardinality of an edge dominating set (resp. a total edge dominating set) of $G$ is the {\em edge domination number} (resp. {\em total edge domination number}) of $G$, denoted by $\gamma^{'}(G)$ (resp. $\gamma_t^{'}(G)$). In the present paper, we prove that the total edge domination problem is NP-complete for bipartite graphs with maximum degree 3. We also design a linear-time algorithm for solving this problem for trees. Finally, for a graph $G$, we give the inequality $\gamma^{'}(G)\leqslant \gamma^{'}_{t}(G)\leqslant 2\gamma^{'}(G)$  and characterize the trees $T$ which obtain the upper or lower bounds in the inequality.

\vspace{1ex}


\section{Introduction}
\rm\normalsize
Dominating problems have been subject of many studies in graph theory, and have many applications in operations research, e.g., in resource allocation and network routing, as well as in coding theory. There are many variants of domination, we mainly fucus on the total edge domination  which is a variant of edge domination. Edge domination is introduced by Mitchell and Hedetniemi $\cite{mh77}$ and is related to telephone switching network $\cite{l68}$. Edge domination is also related to the approximation of the vertex cover problem, since an independent edge dominating set is a matching $\cite{k72}$.

In this paper we in general follow $\cite{hhs98}$ for natation and graph theory terminology. All graphs considered here are finite, undirected, connected, have no loops or multiple edges. Let $G=(V, E)$ be a graph with vertex set $V$ and edge set $E$. A subset $F$ of $E$ is called an $\emph{edge dominating set}$ (abbreviated for {\em ED-set}) of $G$ if every edge not in $F$ is adjacent to at least one edge in $F$. The edge domination number, denoted by $\gamma ^{'}(G)$, is the minimum cardinality of an ED-set of $G$. An ED-set of $G$ with cardinality $\gamma^{'}(G)$ is called a {\em $\gamma^{'}(G)$-set}. The edge domination problem has been studied by several authors for example $\cite{hk93, k98, x06, yg80}$.
Yannakakis and Gavril $\cite{yg80}$ showed that, the edge domination problem is NP-complete even when graphs are planar or bipartite of maximum degree 3, but solvable for trees and claw-free chordal graphs.

The concept of the total edge domination, a variant of edge domination, was introduced by Kulli and Patwari $\cite{kp91}$. A subset $F_t$ of $E$ is called a $\emph{total edge dominating set}$ (abbreviated for {\em TED-set}) of $G$ if every edge  is adjacent to at least one edge in $F_t$. The {\em total edge domination number}, denoted by $\gamma^{'}_{t}(G)$, is the minimum cardinality of a TED-set of $G$. A TED-set of $G$ with cardinality $\gamma^{'}_{t}(G)$ is called a {\em $\gamma^{'}_{t}(G)$-set}. Zhao et al. proved \cite{zlm14} that the total edge domination problem is NP-complete for planar graphs with maximum degree three, and for undirected path graphs and also constructed a linear algorithm for total edge domination problem in trees by a label method. For more study on total edge domination, see for example references  $\cite{ms13, pc16, v14}$.

As far as we know, there is no discussion on the complexity of total edge domination problem for bipartite graphs. For this reason, we prove that the total edge domination problem is NP-complete for bipartite graphs with maximum degree 3. We also design another linear time algorithm for computing $\gamma_t^{'}(T)$ of a tree $T$ by the dynamic programming method, different from the algorithm in $\cite{zlm14}$.  Kulli et al. $\cite {kp91}$ gave the lower bound of the total edge domination number for a graph $G$: $\gamma^{'}(G)\leqslant \gamma^{'}_{t}(G)$,  it is obvious that $\gamma^{'}_{t}(G)\leqslant 2\gamma^{'}(G)$. So, for any graph $G$,  $\gamma^{'}(G)\leqslant \gamma^{'}_{t}(G)\leqslant 2\gamma^{'}(G)$. In this paper, we show that the bounds are sharp and characterize trees achieving the lower or upper bound.

\textbf{Notation.} Let $G=(V, E)$ be a graph. For $v\in V$,  denote by $N_{G}(v)$ the {\em open neighborhood} of $v$ in $G$, i.e., $N_{G}(v)=\{u\in V|~uv \in E\}$,  by $deg_G(v)$ the size of $N_{G}(v)$ called the {\em degree} of $v$, and by $E_{G}(v)$ the set of all the edges of $G$ incident with $v$, i.e., $E_{G}(v)=\{e\in E |$ $v$ is incident with $e\}$. Similarly, for $e\in E$, denote by $N_{G}(e)$ the {\em open neighbourhood} of $e$ in $G$, i.e., $N_{G}(e)=\{e'\in E|$ $e'$ is adjacent to $e \}$ and by $N_{G}[e]=N_{G}(e)\cup \{e\}$ the {\em closed neighbourhood} of $e$. For two vertices $u, v\in V$, the {\em distance} $d_{G}(u,v)$ is defined as the length of a shortest path between $u$ and $v$ in $G$. We define the shorter distance between vertex $w$ and one endpoint of edge $e$ as the {\em distance} between $w$ and $e$, denoted by $d_G(w, e)$. The maximum distance among all pairs of vertices is called the $diameter$ of $G$, denoted by $diam(G)$. 
 If there is no ambiguity in the sequel, the subscript in the notation is omitted.

A $leaf$ of a graph $G$ is a vertex  of degree one and a {\em support vertex} (resp. {\em strong support vertex}) of $G$ is a vertex adjacent to a leaf (resp. adjacent to at least two leaves).  A {\em leaf edge} (or {\em pendant edge}) of $G$ is an edge with one leaf as an endpoint. Consider one vertex of a tree as special, called the {\em root} of this tree. A tree with the fixed root is a {\em rooted tree}. For a vertex $v$ of a rooted tree $T$ with root $r$, a neighbour of $v$ away from $r$ is called a {\em child}. For a positive integer $k$,  a {\em star} $S_{1, k}$ is a tree that contains exactly one non-leaf vertex called a {\em center vertex} and $k$ leaves. A {\em double star} is a tree that contains exactly two non-leaf vertices called {\em center vertices}.

\section{ The result on NP-completeness}

In this section, we are going to prove that the total edge domination problem is NP-complete for bipartite graphs with maximum degree 3. To prove that a problem $P$ is NP-complete, it is enough to prove that $P\in\mathcal{NP}$ and to show that a known NP-complete problem is reducible to the problem $P$ in polynomial time. The known NP-complete problem used in our reduction is the SAT-3 restricted problem as follows:

\vskip 0.1in

\textrm{\bf SAT-3 RESTRICTED PROBLEM (SAT-3 RES)} \cite{y81}.

\noindent{\bf Instance:} A set of clauses $C_{1}, C_{2},\ldots, C_{p}$ containing only variables, with at most three literals per clause, such that every variable occurs two times and its negation once.

\noindent{\bf Question:} Is there a truth assignment of zeros and ones to the variables satisfying all the clauses?

\vskip 0.1in

The decision total edge domination problem is stated as follows:

\vskip 0.1in

\noindent\textbf{Instance:} A graph $G = (V, E)$ and a positive integer $k\leqslant |E|$.

\noindent \textbf{Question:} Does $G$ have a total edge dominating set of size at most $k$?

\vskip 0.1in

Now we can state our main result in this section.

\begin{theorem}\label{th:NPC}
The total edge domination problem for bipartite graphs with maximum degree 3 is NP-complete.
\end{theorem}

\begin{proof}
The reduction is from the SAT-3 restricted problem. Consider a set of clauses $\{C_{1},\ldots, C_{p}\}$ with variables $x_{1}, x_{2},\ldots, x_{n}$ as input for the SAT-3 restricted problem. Now we construct a graph $G=(V, E)$. For any $1\leqslant l\leqslant p$, there are two adjacent vertices, say $d_{l}$ and $d'_{l}$, corresponding to the clause $C_l$, denoted by $G_l$. For any $1\leqslant i\leqslant n$, there is a subgraph of $G$, which is a disjoint union of three paths $a_ia_{i, 0}a_{i, 1}a_{i, 2}$, $b_ib_{i, 0}b_{i, 1}b_{i, 2}$, $c_ic_{i, 0}c_{i, 1}c_{i, 2}$, and two edges $a_ic_i$ and $c_ib_i$, corresponding to the variable $x_i$, denoted by $G_{x_i}$ (see Fig. \ref{Fig:NPCConG}). For any clause $C_{l}$, if $x_i\in C_{l}$, then we connect $d_{l}$ to one of vertices $a_{i, 0}$ and $b_{i, 0}$ to ensure that $d(a_{i, 0})=3$ and $d(b_{i, 0})=3$ (from conditions in SAT-3 RES); if $\overline{x_i}\in C_{l}$, then we connect $d'_{l}$ to $c_{i, 0}$ (for an example, see Fig. 1). It is obvious that $G$ is bipartite, coloring vertices with white and black, shown as Fig. \ref{Fig:NPCConG}.  We will show that there is a truth assignment of zeros and ones to the variables satisfying all clauses $\{C_{1}, C_{2},\ldots, C_{p}\}$ if and only if $G$ has a total edge dominating set of size $6n$.

\begin{figure}[h]
\centering
\includegraphics[scale=0.35]{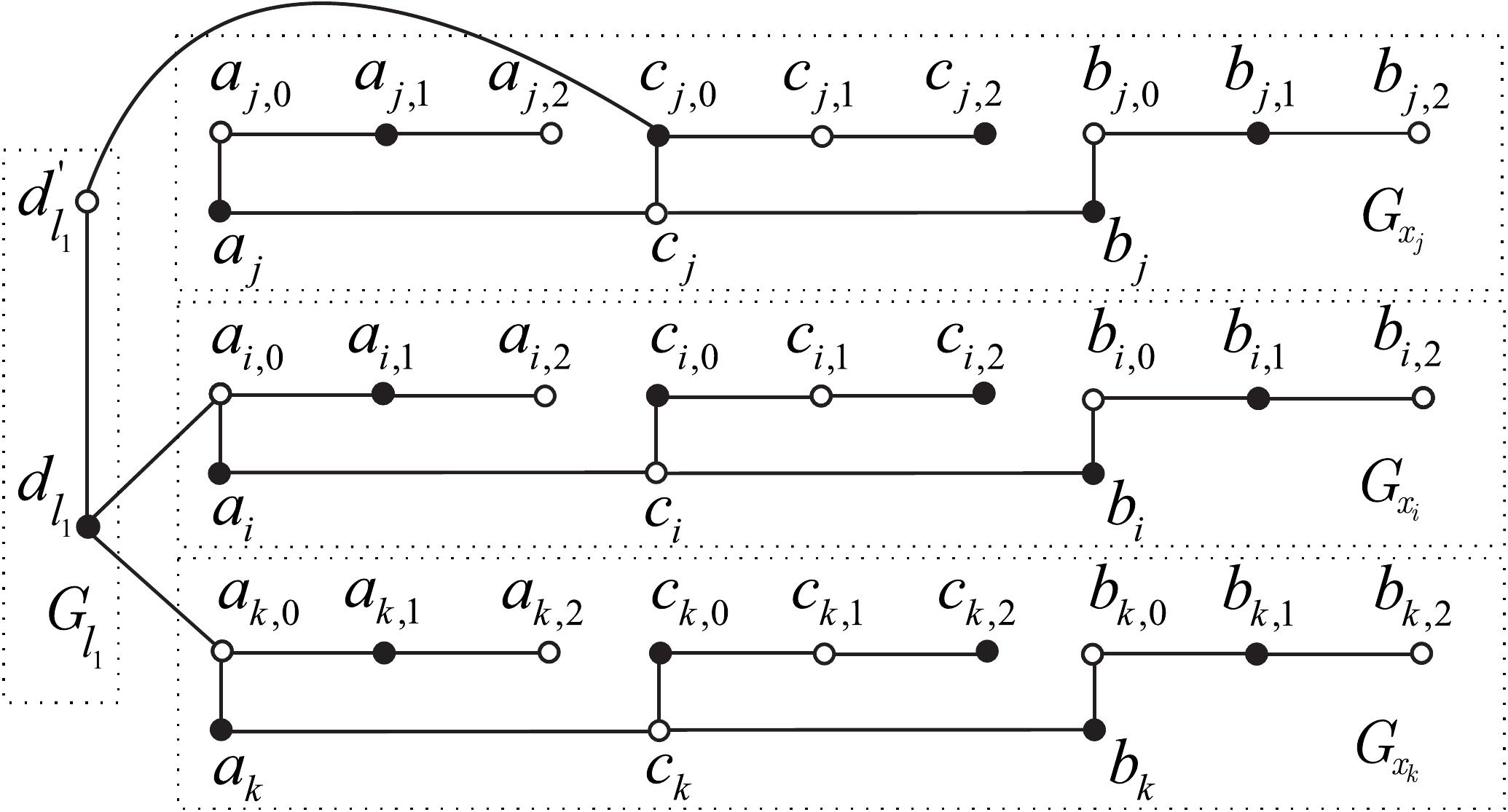}
\caption{The induced subgraph of $G$ by $G_{l_1}$, $G_{x_i}$, $G_{x_j}$ and $G_{x_k}$; for example, $C_{l_1}=(x_i\overline{x_j}x_k)$.}
\vspace{-0.3cm}
\label{Fig:NPCConG}
\end{figure}

Necessity: Given a satisfying assignment of the clauses, define a set $F$ of edges as follows (assume that $x_i$ is in two clauses $C_{l_1}$ and $C_{l_3}$, and $\overline{x_i}$ is in clause $C_{l_2}$):
\setlength{\abovedisplayskip}{1pt}
\setlength{\belowdisplayskip}{1.5pt}
\begin{align*}
F=&\{a_{i, 0}d_{j_{1}}, a_{i, 0}a_{i, 1}, b_{i, 0}d_{j_{3}}, b_{i, 0}b_{i, 1}, c_ic_{i, 0}, c_{i, 0}c_{i, 1}\mid x_{i}=1\}\\
& \cup \{a_ia_{i, 0}, a_{i, 0}a_{i, 1}, b_ib_{i, 0}, b_{i, 0}b_{i, 1}, c_{i, 0}d'_{j_2}, c_{i, 0}c_{i, 1}\mid x_{i}=0\},
\end{align*}
(see Fig. \ref{Fig:NPCConTEDS}). It is obvious that $F$ is a TED-set of size $6n$.

\begin{figure}[h]\label{Fig:NPCConTEDS}
\centering
\subfigure[\scriptsize{In the case $x_i=0$};]
{
    \begin{minipage}[t]{0.4\linewidth}
        \label{a}
        \includegraphics[scale=0.45]{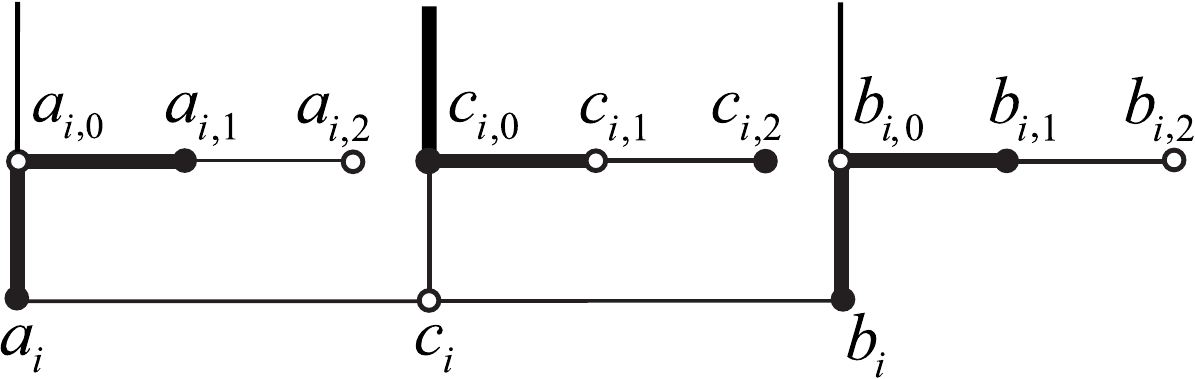}
    \end{minipage}
}
\subfigure[\scriptsize{In the case $x_i=1$.}]{
    \begin{minipage}[t]{0.4\linewidth}
        \label{b}
        \includegraphics[scale=0.5]{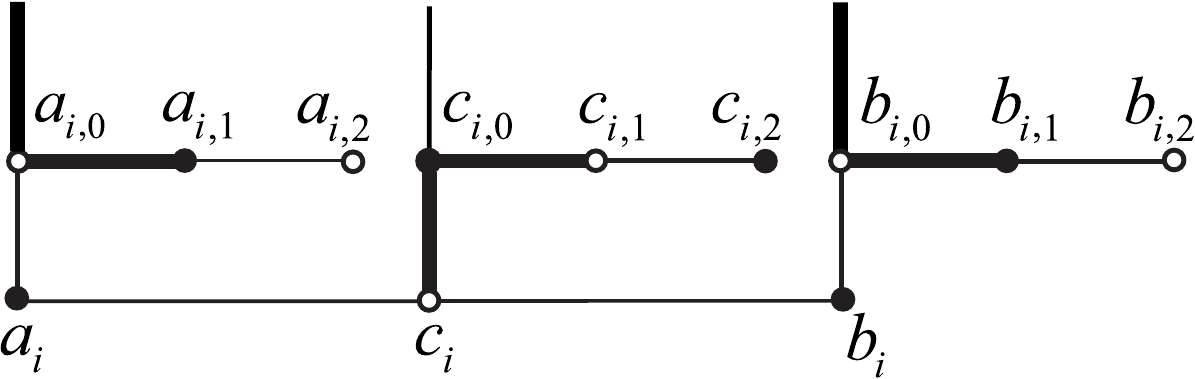}
    \end{minipage}
}
\centering
\caption{The construction of a TED-set $F$ on $G_{x_i}$, represented by the thick edges.}
\vspace{-0.3cm}
\end{figure}

Conversely, we assume that $G$ has a TED-set $F$ of size  $6n$. For any $1\leqslant i\leqslant n$, in view of leaf edges $a_{i, 1}a_{i, 2}, b_{i, 1}b_{i, 2}, c_{i, 1}c_{i, 2}$, $F$ must contain three edges $a_{i, 0}a_{i, 1}, b_{i, 0}b_{i, 1}, c_{i, 0}c_{i, 1}$ and its respective adjacent edges. Thus the subgraph $G_{x_i}$ contains exactly 6 edges in $F$. For the convenience of proof, we assume that $x_i$ is contained in clauses $C_{l_1}$ and $C_{l_2}$, and $\overline{x_i}$ is contained in clause $C_{l_3}$.

\noindent{\bf Case 1.} $c_ic_{i, 0}\not\in F$.

In this case, $F$ must contain $a_ia_{i, 0}, b_ib_{i, 0}$ and we may assume that the edge adjacent to $c_{i, 0}c_{i, 1}$ in $F$ is $c_{i, 0}d'_{l_3}$, otherwise we can add  $c_{i, 0}d'_{l_3}$ into $F$ by deleting $c_{i, 1}c_{i, 2}$ from $F$.

\noindent{\bf Case 2.} $c_ic_{i, 0}\in F$.

Similar to Case 1, we can assume that $a_{i, 0}d_{l_1}, b_{i, 0}d_{l_2}\in F$.

Therefore, regardless of whether $F$ contains $c_ic_{i, 0}$, we can always give a special total edge dominating set $F$ of size $6n$. We define a truth assignment $\tau$ by, if $c_ic_{i, 0}\in F$, setting $x_i=1$ and $x_i=0$, otherwise. Since $F$ is a TED-set constructed as above, at least one edge in $F$ is adjacent to $d_{l}d'_{l}$ for every $l$ (note that $d_{l}d'_{l}\notin F$). Consequently $\tau$ satisfies all clauses.

The degree of vertices except for $d_l$ and $d'_l$ in $G$ constructed above is at most 3, but if $C_l=x_{i_1}x_{i_2}x_{i_3}$ (resp., $\overline{x_{i_1}}~\overline{x_{i_2}}~\overline{x_{i_3}}$), then $d_{G}(d_l)=4$ (resp., $d_{G}(d'_l)=4$). Then we use a tricky technique: (1) replace $H$ shown as Fig. \ref{Fig:NPCDeg4a} for $d_ld'_l$ and, (2) replace the three edges connecting the vertices $a, b, c$ corresponding to variables and $d_l$ (resp. $d'_l$) with the three edges connecting  $a, b, c$ and $x, y, z$ in $H$, respectively, say $ax, by, cz$.

\begin{figure}[h]
\centering
\subfigure[\scriptsize{$H$ contains exactly 9 edges in $F$ when none of the three edges $\{ax, by, cz\}$ belongs to $F$;}]{
    \begin{minipage}[t]{0.45\linewidth}
        \centering
        \label{Fig:NPCDeg4a}
        \includegraphics[scale=0.85]{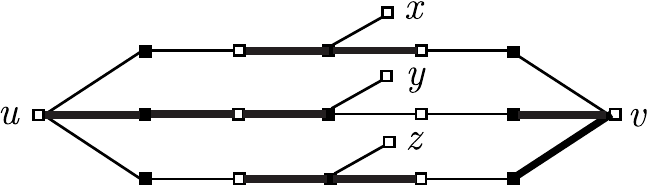}
        \vspace{0.2cm}
    \end{minipage}
}
\subfigure[\scriptsize{ $H$ contains exactly 8 edges in $F$ when $ax$ is in $F$}.]{
    \begin{minipage}[t]{0.45\linewidth}
        \centering
        \label{Fig:NPCDeg4b}
        \includegraphics[scale=0.85]{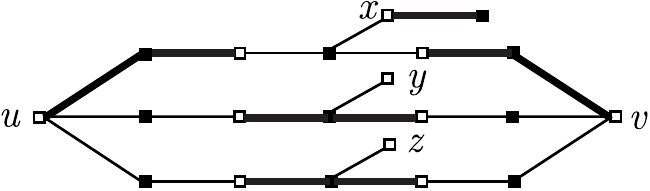}
        \vspace{0.2cm}
    \end{minipage}
}
\centering
\caption{The graph $H$ and the construction of a TED-set $F$ on $H$, represented by the
thick edges.}
\vspace{-0.3cm}
\end{figure}

It is easy to show by a straightforward case analysis that: for  a TED-set $F$ of $G$, \\
(1). if none of the three edges $\{ax, by, cz\}$ belongs to $F$, then $F$ contains at least nine edges from $H$, see Fig. \ref{Fig:NPCDeg4a}.\\
(2). if one of three edges $\{ax, by, cz\}$  is in $F$, say $ax$, then $F$ contains at least eight edges from  $H$, see Fig. \ref{Fig:NPCDeg4b}.

Especially, let $s$ be the number of 3-literal clauses which satisfies that the literals contained are all positive or all negative. Then we can similarly show that there is a truth assignment of zeros and ones to the variables satisfying all clauses $\{C_{1}, C_{2},\ldots, C_{p}\}$ if and only if $G$ has a total edge dominating set of size $6n+8s$.
\end{proof}

From the proof of Theorem \ref{th:NPC}, the graph constructed has a girth of at least 10.

\begin{corollary}
The total edge domination problem for bipartite graphs of girth at least 10 with maximum degree 3 is NP-complete.
\end{corollary}

\begin{proof}
The notations are as in the proof of Theorem \ref{th:NPC}. By the construction of $G$,  there are no edges among $G_{l}$'s (or $H$) and among $G_{x_i}$'s. So a cycle $C$ is either in $H$ ( note that there is
no cycles in $G_l$ or $G_{x_i}$)or formed by going through $G_{l_1}$, $G_{x_{i_1}}$, $G_{l_2}$, $G_{x_{i_2}}$, $\ldots$,  $G_{l_k}$, $G_{x_{k}}$, $G_{l_1}$ $(k\geqslant 2)$; in the second case the intersection of $C$ and $G_{x_i}$ contains at least three edges and so the length of $C$ is at least $5k\geqslant 10$. Note that the girth of $H$ is more than 12.
\end{proof}

\section{A linear-time algorithm for trees}
In this section, we work on a linear-time algorithm for finding the total edge domination number of a tree by using the dynamic programming method.

First, we define some sets and some parameters. Let $T$ be a tree with an edge $e$.  We define:
\setlength{\abovedisplayskip}{1pt}
\setlength{\belowdisplayskip}{1.5pt}
\begin{align*}
\mathcal{F}_{1}(T, e) :=  \{& F|~F \text{ is a  TED-set of }  T  \text{ with }  e \in F\};\\
\mathcal{F}_{0}(T, e) := \{ & F|~F \text{ is a TED-set of } T \text{ with } e \notin F\} ;\\
\mathcal{F}_{\overline{1}}(T, e) := \{& F|~F \text{ is an ED-set of } T \text{ with a unique isolated edge $e$ in $F$}\};\\
\mathcal{F}_{\overline{0}}(T, e) := \{& F|~F \text{ is a TED-set of } T-e, \text{but $e$ is not dominated by $F$}\}.%
\end{align*}

It is easily obtained
\begin{lemma}\label{Le:foursetsnonEmp}
Let $e$ be a leaf edge of  tree  $T$. Then\\
$\mathcal{F}_1(T, e)\neq \emptyset$ if and only if $T\neq K_2$;\\
$\mathcal{F}_0(T, e)\neq \emptyset$ if and only if $T$ has at least 3 edges;\\
$\mathcal{F}_{\overline{1}}(T, e)\neq \emptyset$ (resp. $\mathcal{F}_{\overline{0}}(T, e)\neq \emptyset$) if and only if $T\setminus N[e]$ has no $K_2$ as components.
\end{lemma}

We denote
\begin{align*}
\gamma'_1(T, e): =& \mbox{min} \{~{|F|}~\big|~ F \in \mathcal{F}^{}_{1}(T,e)\};\\
\gamma'_0(T, e):=& \mbox{min} \{{~|F|}~\big|~ F \in \mathcal{F}_{0}(T,e)\};\\
\gamma'_{\overline{1}}(T, e): =& \mbox{min} \{~{|F|}~\big|~ F\in \mathcal{F}_{\overline{1}}(T,e)\};\\
\gamma'_{\overline{0}}(T,e):=& \mbox{min} \{~{|F|}~\big|~F\in\mathcal{F}_{\overline{0}}(T,e)\}.
\end{align*}

By convention, if a set is empty, then we set the value as infinity. For example, if $\mathcal{F}_{\overline{0}}(T,e)= \emptyset$, then  we set $\gamma'_{\overline{0}}(T,e)=\infty$.
We can define  $F\in \mathcal{F}_{1}(T, e)$ (resp. $\mathcal{F}_{0}(T, e)$, $\mathcal{F}_{\overline{1}}(T, e)$, $\mathcal{F}_{\overline{0}}(T, e)$) of minimum cardinality as a $\gamma'_{1}(T, e)$ (resp. $\gamma'_{0}(T, e)$, $\gamma'_{\overline{1}}(T, e)$, $\gamma'_{\overline{0}}(T, e)$){\em-set} of $T$. We give some inequality relationships among four values defined as above.

\begin{lemma}\label{ineq:ParaRela}
Let $T$ be a tree with an edge $e$. If $\mathcal{F}_{1}(T, e), \mathcal{F}_{0}(T, e), \mathcal{F}_{\overline{1}}(T, e)$ and $\mathcal{F}_{\overline{0}}(T,e)$ are non-empty sets, then\\
(1) $\gamma'_{1}(T,e) \leqslant \gamma'_{0}(T,e)+1$;\\
(2) $\gamma'_{1}(T,e) \leqslant \gamma'_{\overline{1}}(T,e)+1$;\\
(3) $\gamma'_{1}(T,e)\leqslant \gamma'_{\overline{0}}(T,e)+2$;\\
(4) $\gamma'_{\overline{1}}(T,e) \leqslant \gamma'_{\overline{0}}(T,e)+1$.
\end{lemma}

\begin{proof} Let $e'$ be any edge in $N(e)$.

(1) Let $F\in \mathcal{F}_{0}(T,e)$. Then there exists an edge $e''\in F$ adjacent to $e$ and further $F+e$ is a TED set of $T$ containing $e$. Therefore $\gamma'_{1}(T,e) \leqslant \gamma'_{0}(T,e)+1$.

(2) Let $F_{0}\in \mathcal{F}_{\overline{1}}(T,e)$. Then $e\in F_{0}$ and $N(e)\cap F_{0}=\emptyset$ by the definition of $\mathcal{F}_{\overline{1}}(T,e)$. $F_{0}+e'$ is a TED-set of containing $e$. Therefore $\gamma'_{t,1}(T,e)\leqslant \gamma'_{\overline{1}}(T,e)+1$.

(3) Let $F_{1}\in \mathcal{F}_{\overline{0}}(T,e)$. Then $N[e]\cap F_{1}= \emptyset$ by the definition of $\mathcal{F}_{\overline{0}}(T,e)$. $F_{1}+e +e'$ is a TED-set of $T$ containing $e$. Thus $\gamma'_{1}(T,e)\leqslant \gamma'_{t,\overline{0}}(T,e)+2$.

(4) Let $F_{2}\in \mathcal{F}_{\overline{0}}(T,e)$. Then $F_2+ e$ is an ED-set of $T$ with a unique isolated edge $e$ by the definition of $\mathcal{F}_{\overline{1}}(T,e)$. Thus $\gamma'_{\overline{1}}(T,e)\leqslant \gamma'_{t,\overline{0}}(T,e)+1$.
\end{proof}

Before giving the dynamic programming algorithm, we designed an edge data structure as follows.

Root the tree $T$ at any leaf, say $r$. The {\em height}, denoted by $h$, of $T$ is the maximum distance between $r$ and all other vertices of $T$. The {\em level} $i$ $(0\leqslant i\leqslant h)$ is the set of vertices of $T$ with a distance $i$ from $r$.

For such a rooted tree $T$ of order $n+1$, let us label the edges of $T$ as $1, 2, \ldots, n$. We go through every level from  $h$ to 1. For each $i$, $1\leqslant i\leqslant h$, we traverse the edges connecting the vertices on $i$ and $i-1$ in any order,  from left to right. We list the fathers of all edges of $T$ (the edge numbered $n$ has no father by writing father $[n]=0$ ), so we can use a data structure called an {\em edge parent array} to represent $T$. 
Let $e^0$ be a non-leaf edge in rooted tree $T$, $u$ the endpoint of $e^0$ away from the root. Denote by $N_c(e^0)$ the set of neighbors of $e^0$ with endpoints $u$, called {\em children neighbors} of $e^0$, say  $\{e^1, e^2, \ldots, e^q\}$ for some integer $q$. For $0\leqslant j\leqslant q$, let $T^j$ be the component containing $e^j$ of $T\setminus (\{e^0, e^1, \ldots, e^q\}\setminus \{e^j\})$.

\begin{theorem}
Let $T$ be a rooted tree with a non-leaf edge $e^0$ and $N_c(e^0)=\{e_1, e_2, \ldots, e_q\}$ for some integer $q \geqslant 1$. For $0\leqslant j\leqslant q$, $T^j$ are defined as above, and denote
\setlength{\abovedisplayskip}{1pt}
\setlength{\belowdisplayskip}{1.5pt}
\begin{align*}
\theta_{j}:=&\min\{\gamma'_{1}(T^{j}, e^{j}),\gamma'_{0}(T^{j}, e^{j}), \gamma'_{\overline{1}}(T^{j},e^{j}),\gamma'_{\overline{0}}(T^{j}, e^{j})\};\\
A_1:=&\{j\in\{1,2,\ldots,q\}|\theta_{j}=\gamma'_{1}(T^{j}, e^{j}) \};\\
A_2:=&\{j\in\{1,2,\ldots,q\}|\theta_{j}=\gamma'_{0}(T^{j}, e^{j}) \};\\
A_3:=&\{j\in\{1,2,\ldots,q\}|\theta_{j}=\gamma'_{\overline{1}}(T^{j}, e^{j}) \};\\
A_4:=&\{j\in\{1,2,\ldots,q\}|\theta_{j}=\gamma'_{\overline{0}}(T^{j}, e^{j}) \}.\\
\end{align*}
Then
\begin{small}
\begin{align*}
(1). &~\gamma'_{1}(T, e^0)=
\begin{cases}
\min\{\gamma'_{1}(T^0, e^0),\gamma'_{\overline{1}}(T^0, e^0)\}+\sum\limits^{q}_{j=1}\theta_{j}, &\substack{\text{if}~A_1\cup A_3\neq\emptyset};\\
\min\{\gamma'_{1}(T^0, e^0), \gamma'_{\overline{1}}(T^0, e^0)+1\}+
 \sum\limits^{q}_{\substack{j=1}}\theta_{j}, &\substack{\text{if} ~A_1\cup A_3= \emptyset.}
\end{cases}
\\
(2). & ~\gamma'_{0}(T, e^0)=
\begin{cases}
\min\{\gamma'_{0}(T^0, e^0), \gamma'_{\overline{0}}(T^0,e^0)\}+\sum\limits^{q}_{j=1}\theta_{j}, &\substack{\text{if}~A_1\neq\emptyset~or ~|A_3|\geqslant 2;}\\
\min\{\gamma'_{0}(T^0, e^0), \gamma'_{\overline{0}}(T^0, e^0)\}+
 \sum\limits^{q}_{\scriptsize{j=1}}\theta_{j}+1,& \substack{\text{if}~ A_1=\emptyset~and~|A_3|=1 ~or~\\ A_1=A_3=\emptyset, ~A_2\neq \emptyset, ~A_4\neq \emptyset;}\\
\min\{\gamma'_{0}(T^0, e^0),\gamma'_{\overline{0}}(T^0, e^0)+1\}+\sum\limits^{q}_{\substack{j=1}}\theta_{j}, &\substack{\text{if } A_1=A_3=A_4=\emptyset;}\\
\min\{\gamma'_{0}(T^0, e^0),\gamma'_{\overline{0}}(T^0, e^0)\}+\sum\limits^{q}_{j=1}\theta_{j} +1,&\substack{\text{if}~ A_1=A_2=A_3=\emptyset,~ and~there~is~ j\in A_4\\~such~that~ \gamma'_{1}(T^{j},e^{j})-\gamma'_{\overline{0}}(T^{j},e^{j})=1; }\\
\min\{\gamma'_{0}(T^0, e^0),\gamma'_{\overline{0}}(T^0, e^0)\}+\sum\limits^{q}_{j=1}\theta_{j} +2,
& \substack{\text{if}~ A_1=A_2=A_3=\emptyset~and~ any~j\in A_4,\\ \gamma'_{1}(T^{j},e^{j})-\gamma'_{\overline{0}}(T^{j},e^{j})=2. }\\
\end{cases}
\\
(3). &~\gamma'_{\overline{1}}(T, e^0)= \gamma'_{\overline{1}}(T^0, e^0)+\sum\limits^{q}_{j=1} \min \{\gamma'_{0}(T^{j}, e^{j}), \gamma'_{\overline{0}}(T^{j}, e^{j})\};
\\
(4). &~\gamma'_{\overline{0}}(T, e^0)= \gamma'_{\overline{0}}(T^0,e^0)+\sum\limits^{q}_{j=1} \gamma'_{0}(T^{j}, e^{j}).
\end{align*}
\end{small}
\end{theorem}

\begin{proof}
For the convenience, for $0\leqslant j\leqslant q$, we define $F_{T^j}=F_T\cap T^j$ for an edge subset $F_T$ of $T$ and thus $|F_T|=\sum_{j=0}^q |F_{T^j}|$. Especially, for $0\leqslant j\leqslant q$,  if $F_T$ is a TED-set of $T$, then $F_{T^j}\in \mathcal{F}_{1}(T^{j},e^{j})\cup \mathcal{F}_{0}(T^{j},e^{j})\cup
\mathcal{F}_{\overline{1}}(T^{j},e^{j})\cup
\mathcal{F}_{\overline{0}}(T^{j},e^{j})$ by the definition. Denote $\overline{N_c}(e^0)=N(e^0)\setminus N_c(e^0)$.

(1). Let $F_{T}$ be a $\gamma'_{ 1}(T, e^0)$-set.

\noindent{\bf Case 1.1.} $\overline{N_c}(e^0)\cap F_{T}\neq \emptyset$.

In this case, the restriction $F_{T^0}$ of $F_T$ on $T^0$ is a TED-set of $T^0$, further a $\gamma'_{ 1}(T^0,e^0)$-set. For any $j$ ($1\leqslant j\leqslant q$), $F_{T^j}$ is a set of size $\theta_j$ in $\mathcal{F}_{1}(T^{j},e^{j})\cup \mathcal{F}_{0}(T^{j},e^{j})\cup
\mathcal{F}_{\overline{1}}(T^{j},e^{j})\cup
\mathcal{F}_{\overline{0}}(T^{j},e^{j})$ by the definition of $F_{T^j}$. So
\setlength{\abovedisplayskip}{1pt}
\setlength{\belowdisplayskip}{1.5pt}
\begin{align*}
\gamma'_{1}(T, e^0)= \gamma'_{1}(T^0, e^0)+ \sum\limits^{q}_{j=1}\theta_{j}.
\end{align*}

\noindent{\bf Case 1.2.} $\overline{N_c}(e^0)\cap F_{T}= \emptyset$.

In this case, $F_{T^0}\in\mathcal{F}_{\overline{1}}(T^0, e^0)$. Thus
\setlength{\abovedisplayskip}{1pt}
\setlength{\belowdisplayskip}{1.5pt}
\begin{equation}\label{ineq:case12}
\gamma'_{1}(T, e^0)\geqslant \gamma'_{ \overline{1}}(T^0, e^0)+\sum_{j=1}^q \theta_j.
\end{equation}
In order to connect $e^0$ in $F_T$, there exists some $1\leqslant j\leqslant q$ such that $e^j\in F_{T^j}$.

\noindent{\bf Subcase 1.2.1.} $A_1\cup A_3\neq \emptyset$, say, $j_1\in A_1$.

We take any $\gamma'_{ \overline{1}}(T^0, e^0)$-set $B^0$ and $\gamma'_{1}(T^{j_1}, e^{j_1})$-set $B^{j_1}$. For any $j\neq j_1$ $(1\leqslant j\leqslant q)$, we choose an edge set $B^j$ of size $\theta_j$ in $\mathcal{F}_{1}(T^{j},e^{j})\cup \mathcal{F}_{0}(T^{j},e^{j})\cup
\mathcal{F}_{\overline{1}}(T^{j},e^{j})\cup
\mathcal{F}_{\overline{0}}(T^{j},e^{j})$.  Then $\cup_{j=0}^q B^j$ is a TED-set of $T$ of size $\gamma'_{ \overline{1}}(T^0, e^0)+\sum_{j=1}^q \theta_j$ satisfying $|N
(e_0)\cap (\cup_{j=0}^q B^j)|\geqslant1$. Combined with \eqref{ineq:case12}, we have

\setlength{\abovedisplayskip}{1pt}
\setlength{\belowdisplayskip}{1.5pt}
\begin{align*}
\gamma'_{1}(T, e^0)= \gamma'_{\overline{1}}(T^0, e^0)+
\sum\limits^{q}_{j=1}\theta_{j}.
\end{align*}

\noindent{\bf Subcase 1.2.2.} $A_1\cup A_3=\emptyset$, i.e., $A_1=\emptyset$ and $A_3=\emptyset$.

In this subcase, equality does not hold in Eq. \eqref{ineq:case12}. If $A_2\neq \emptyset$, combined with Lemma \ref{Le:foursetsnonEmp}, Lemma \ref{ineq:ParaRela} (1) and $A_1=\emptyset$, for any $j\in A_2$, 
$\gamma'_{1}(T^j, e^j)= \gamma'_{0}(T^j, e^j)+1=\theta_j+1$. Otherwise, If $A_2=\emptyset$, then $A_4=\{1, 2, \ldots, q\}(\neq \emptyset).$ Combined with Lemma \ref{ineq:ParaRela} (4) and $A_3=\emptyset$, for any $j\in A_4$, $\gamma'_{\overline{1}}(T^{j}, e^{j})= \gamma'_{\overline{0}}(T^{j}, e^{j})+1=\theta_j+1$.  Similar to  Subcase 1.2.1, whatever which case it is, we can construct a TED-set of $T$ of size $\gamma'_{\overline{1}}(T^0, e^0)+\sum_{j=1}^q \theta_j+1$ satisfying $|N(e_0)\cap (\cup_{j=0}^q B^j)|\geqslant1$. So
\setlength{\abovedisplayskip}{1pt}
\setlength{\belowdisplayskip}{1.5pt}
 \begin{align*}
 \gamma'_{1}(T, e^0)=
 \gamma'_{\overline{1}}(T^0, e^0)+
 \sum\limits^{q}_{\substack{j=1}}\theta_{j}+1.
 \end{align*}

(2). Let $F_{T}$ be a $\gamma'_{ 0}(T, e^0)$-set. \\

If $\overline{N_c}(e^0)\cap F_{T}\neq \emptyset$, then the restriction $F_{T^0}$ of $F_T$ on $T^0$ is a TED-set of $T^0$, further a $\gamma'_{0}(T^0, e^0)$-set. So
\setlength{\abovedisplayskip}{1pt}
\setlength{\belowdisplayskip}{1.5pt}
\begin{equation}\label{ineq:case01}
\gamma'_{0}(T, e^0)\geqslant \gamma'_{0}(T^0, e^0)+\sum_{j=1}^q \theta_j.
\end{equation}

If  $\overline{N_c}(e^0)\cap F_{T}= \emptyset$, then the restriction $F_{T^0}$ of $F_T$ on $T^0$ belongs to $\mathcal{F}_{\overline{0}}(T^0, e^0)$, further a $\gamma'_{\overline{0}}(T^0, e^0)$-set. So
\setlength{\abovedisplayskip}{1pt}
\setlength{\belowdisplayskip}{1.5pt}
\begin{equation}\label{ineq:case02}
\gamma'_{0}(T, e^0)\geqslant \gamma'_{\overline{0}}(T^0, e^0)+\sum_{j=1}^q \theta_j.
\end{equation}

\noindent{\bf Case 2.1.} $A_1\neq\emptyset$, say $j_{1}\in A_1$.

We take any $\gamma'_{0}(T^0, e^0)$-set in the case of $\overline{N_c}(e^0)\cap F_{T}\neq \emptyset$ and any $\gamma'_{\overline{0}}(T^0, e^0)$-set in the case of $\overline{N_c}(e^0)\cap F_{T}= \emptyset$. Denoted by $B^0$, any $\gamma'_{1}(T^{j_1}, e^{j_1})$-set $B^{j_1}$, and for any $j\neq j_1$ $(1\leqslant j\leqslant q)$, an edge set $B^j$ of size $\theta_j$ in $\mathcal{F}_{1}(T^{j},e^{j})\cup \mathcal{F}_{0}(T^{j},e^{j})\cup
\mathcal{F}_{\overline{1}}(T^{j},e^{j})\cup
\mathcal{F}_{\overline{0}}(T^{j},e^{j})$.  Thus $\cup_{j=0}^q B^j$ is a TED-set of $T$ of size $\gamma'_{0}(T^0, e^0)+\sum_{j=1}^q \theta_j$ in the case of $\overline{N_c}(e^0)\cap F_{T}\neq \emptyset$ or $\gamma'_{\overline{0}}(T^0, e^0)+\sum_{j=1}^q \theta_j$ in the case of $\overline{N_c}(e^0)\cap F_{T}= \emptyset$ satisfying $|N(e_0)\cap (\cup_{j=0}^q B^j)|\neq 0$. Combined with \eqref{ineq:case01} and \eqref{ineq:case02}, we have
\setlength{\abovedisplayskip}{1pt}
\setlength{\belowdisplayskip}{1.5pt}
\begin{align*}
\gamma'_{0}(T, e^0)=\mbox{min}\{\gamma'_{0}(T^0, e^0), \gamma'_{\overline{0}}(T^0, e^0)\}+ \sum\limits^{q}_{j=1}\theta_{j}.
\end{align*}

\noindent{\bf Case 2.2.} $A_1=\emptyset$ and $A_3\neq \emptyset$.

If $|A_3|\geqslant 2$, then, for $j_1, j_2\in A_2$, we can take a $\gamma'_{\overline{1}}(T^{j_1}, e^{j_1})$-set $B^{j_1}$ and a $\gamma'_{\overline{1}}(T^{j_2}, e^{j_2})$-set $B^{j_2}$. The others $B^0$ and $B^j$ for $1\leqslant j\leqslant q$ and $j\neq j_1, j_2$ are taken as Subcase 2.1. Similarly,  we can obtain
\setlength{\abovedisplayskip}{1pt}
\setlength{\belowdisplayskip}{1.5pt}
\begin{align*}
\gamma'_{0}(T, e^0)=\mbox{min}\{\gamma'_{0}(T^0, e^0), \gamma'_{\overline{0}}(T^0, e^0)\}+ \sum\limits^{q}_{j=1}\theta_{j}.
\end{align*}

If $|A_3|=1$, say $A_3=\{j_{3}\}$, then neither Eq. \eqref{ineq:case01} nor Eq. \eqref{ineq:case02} take equality in this case.  According to Lemma \ref{ineq:ParaRela} (2) and $A_1=\emptyset$, $\gamma'_{1}(T^{j_3}, e^{j_3})=\gamma'_{\overline{1}}(T^{j_3}, e^{j_2})+1= \theta_{j_3} + 1$. We take a $\gamma'_{1}(T^{j_3}, e^{j_3})$-set $B^{j_3}$. The others $B^0$ and $B^j$ for $1\leqslant j\leqslant q$ and $j\neq j_3$ are taken as Subcase 2.1. Thus $\cup_{j=0}^q B^j$ is a TED-set of $T$ of size $\gamma'_{0}(T^0, e^0)+\sum_{j=1}^q \theta_j+1$ in the case of $\overline{N_c}(e^0)\cap F_{T}\neq \emptyset$ or $\gamma'_{\overline{0}}(T^0, e^0)+\sum_{j=1}^q \theta_j+1$ in the case of $\overline{N_c}(e^0)\cap F_{T}= \emptyset$ satisfying $|N(e_0)\cap (\cup_{j=0}^q B^j)|\neq 0$. So
\setlength{\abovedisplayskip}{1pt}
\setlength{\belowdisplayskip}{1.5pt}
\begin{align*}
\gamma'_{0}(T, e^0)=\mbox{min}\{\gamma'_{0}(T^0, e^0), \gamma'_{\overline{0}}(T^0, e^0)\}+ \sum\limits^{q}_{j=1}\theta_{j}+1.
\end{align*}

\noindent{\bf Case 2.3.}  $A_1=A_3=\emptyset$ and $A_2\neq \emptyset$.

If $A_4=\emptyset$, i.e., $A_2=\{1, 2, \ldots, q\}$, and $\overline{N_c}(e^0)\cap F_{T}\neq \emptyset$, then we take any $\gamma'_{0}(T^0, e^0)$-set $B^0$. For $1\leqslant j \leqslant q$, we take any $\gamma'_{0}(T^{j}, e^j)$-set $B^j$. Thus $\cup_{j=0}^q B^j$ is a TED-set of $T$ of size $\gamma'_{0}(T^0, e^0)+\sum_{j=1}^q \theta_j$.

If $A_4=\emptyset$ and $\overline{N_c}(e^0)\cap F_{T}= \emptyset$, then equality does not hold in Eq. \eqref{ineq:case02}. By Lemma \ref{ineq:ParaRela} (1) and $A_1=\emptyset$, for any $ j\in A_2, \gamma'_{1}(T^j, e^j)= \gamma'_{0}(T^j, e^j)+1=\theta_j+1$. We take a $\gamma'_{1}(T^{j_1}, e^{j_1})$-set $B^{j_1}$ for some $1\leqslant j_1 \leqslant q$ and others $B^j$ for any $0\leqslant j\leqslant q$ and $j\neq {j_{1}}$ are taken as in Subcase 2.1. Thus $\cup_{j=0}^q B^j$ is a TED-set of $T$ of size $\gamma'_{\overline{0}}(T^0, e^0)+\sum_{j=1}^q \theta_j+1$.

If $A_4\neq 0$, then equality does not hold in Eqs. \eqref{ineq:case01} and \eqref{ineq:case02}. By Lemma \ref{ineq:ParaRela} (1) and $A_1=\emptyset$, for any $ j\in A_2, \gamma'_{1}(T^j, e^j)= \gamma'_{0}(T^j, e^j)+1=\theta_j+1$. We take any $\gamma'_1(T^
{j_{1}}, e^{j_{1}})$-set $B^{j_{1}}$ for some $1\leqslant {j_{1}}\leqslant q$ and the others $B^j$ for any $0\leqslant j\leqslant q$ and $j\neq {j_{1}}$ are taken as in Subcase 2.1. Thus $\cup_{j=0}^q B^j$ is a TED-set of $T$ of size $\gamma'_{0}(T^0, e^0)+\sum_{j=1}^q \theta_j+1$ in the case of $\overline{N_c}(e^0)\cap F_{T}\neq \emptyset$ or $\gamma'_{\overline{0}}(T^0, e^0)+\sum_{j=1}^q \theta_j+1$ in the case of $\overline{N_c}(e^0)\cap F_{T}= \emptyset$.

So
\setlength{\abovedisplayskip}{1pt}
\setlength{\belowdisplayskip}{1.5pt}
 \begin{align*}
 \gamma'_{0}(T, e^0)=
 \begin{cases}
\min\{\gamma'_{0}(T^0, e^0), \gamma'_{\overline{0}}(T^0, e^0)+1\} +\sum\limits^{q}_{j=1}\theta_{j},&\substack{\text{ if } A_4=\emptyset};\\
\min\{\gamma'_{0}(T^0, e^0), \gamma'_{\overline{0}}(T^0, e^0)\}+\sum\limits^{q}_{\substack{j=1}}\theta_{j}+1,&\substack{\text{ if}~A_4\neq\emptyset}.
\end{cases}
 \end{align*}

\noindent{\bf Case 2.4.} $A_1=A_2=A_3=\emptyset$, i.e., $A_4=\{1, 2, \ldots, q\}$.

In this case, to obtain a $\gamma'_0(T, e^0)$-set, we need one $\gamma'_1(T^{j'}, e^{j'})$-set or at least two $\gamma'_{\overline{1}}(T^{j''}, e^{j''})$-sets for $1\leqslant j''\leqslant q$. So, equality does not hold in Eqs. \eqref{ineq:case01} and \eqref{ineq:case02}. By Lemma \ref{ineq:ParaRela} (4), for each $\forall j\in A_4$, we have $\gamma'_{\overline{0}}(T^{j}, e^{j})+1\leqslant\gamma'_{1}(T^{j}, e^{j})\leqslant \gamma'_{\overline{0}}(T^{j}, e^{j})+2$.
If there exists $j_{4}\in A_4$ such that $\gamma'_{1}(T^{j_{4}},e^{j_{4}})-\gamma'_{\overline{0}}(T^{j_{4}},e^{j_{4}})=1$, then we can take a $\gamma'_{1}(T^{j_{4}}, e^{j_{4}})$-set $B^{j_4}$ and the others $B^j$ for $0\leqslant j\leqslant q$ and $j\neq j^4$ are taken as in Subcase 2.1.  Thus $\cup_{j=0}^q B^j$ is a TED-set of $T$ of size $\gamma'_{0}(T^0, e^0)+\sum_{j=1}^q \theta_j+1$ in the case of $\overline{N_c}(e^0)\cap F_{T}\neq \emptyset$ or $\gamma'_{\overline{0}}(T^0, e^0)+\sum_{j=1}^q \theta_j+1$ in the case of $\overline{N_c}(e)\cap F_{T}= \emptyset$. Otherwise, for all $j$, $\gamma'_{1}(T^j, e^{j})-\gamma'_{\overline{0}}(T^j, e^j)=2.$  Thus, the left-hand sides in both Eqs. \eqref{ineq:case01} and \eqref{ineq:case02} are at least two more than the right-hand sides. We can take a $\gamma'_{1}(T^{j_{4}}, e^{j_{4}})$-set $B^{j_4}$ and the others $B^j$ for $0\leqslant j\leqslant q$ and $j\neq j^4$ are taken as in Subcase 2.1.  Thus $\cup_{j=0}^q B^j$ is a TED-set of $T$ of size $\gamma'_{0}(T^0, e^0)+\sum_{j=1}^q \theta_j+2$ in the case of $\overline{N_c}(e^0)\cap F_{T}\neq \emptyset$ or $\gamma'_{\overline{0}}(T^0, e^0)+\sum_{j=1}^q \theta_j+1$ in the case of $\overline{N_c}(e^0)\cap F_{T}= \emptyset$.  Therefore

\setlength{\abovedisplayskip}{1pt}
\setlength{\belowdisplayskip}{1.5pt}
\begin{align*}
 \gamma'_{0}(T, e^0)=
\begin{cases}
\min\{\gamma'_{0}(T^0, e^0), \gamma'_{\overline{0}}(T^0, e^0)\}+\sum\limits^{q}_{\substack{j=1}}\theta_{j}+1,&\substack{there~is~j~such~that~ \gamma'_{1}(T^{j}, e^{j})-\gamma'_{t,\overline{0}}(T^{j},e^j)=1};\\
\min\{\gamma'_{0}(T^0, e^0), \gamma'_{\overline{0}}(T^0, e^0)\}+\sum\limits^{q}_{j=1}\theta_{j}+2,&\substack{for~any~j~such~that~ \gamma'_{1}(T^{j}, e^{j})-\gamma'_{t,\overline{0}}(T^{j},e^j)=2}.
\end{cases}
 \end{align*}

(3). Let $F_{T}$ be a  $\gamma'_{\overline{1}}(T, e)$-set.

 The restriction $F_{T^0}$ of $F_T$ on $T^0$ belongs to $\mathcal{F}_{\overline{1}}(T^0, e)$, for $1\leqslant j\leqslant q$, the restriction $F_{T^j}$ of $F_T$ on $T^j$ belongs to $\mathcal{F}_0(F^j, e^j)$ or $\mathcal{F}_{\overline{0}}(F^j, e^j)$, the converse also holds. Therefore
 \setlength{\abovedisplayskip}{1pt}
\setlength{\belowdisplayskip}{1.5pt}
\begin{align*}
\gamma'_{\overline{1}}(T, e)= \gamma'_{\overline{1}}(T^0, e)+\sum\limits^{q}_{j=1} \mbox{min} \{\gamma'_{0}(T^{j}, e^{j}), \gamma'_{\overline{0}}(T^{j}, e^{j})\}.
\end{align*}

(4). Let $F_{T}$ be a  $\gamma'_{\overline{0}}(T, e)$-set.\\
The restriction $F_{T^0}$ of $F_T$ on $T^0$ belongs to $\mathcal{F}_{\overline{0}}(T^0, e)$, for $1\leqslant j\leqslant q$, the restriction $F_{T^j}$ of $F_T$ on $T^j$ belongs to $\mathcal{F}_0(F^j, e^j)$, the converse also holds. Therefore
\setlength{\abovedisplayskip}{1pt}
\setlength{\belowdisplayskip}{1.5pt}
\begin{align*}
\gamma'_{\overline{0}}(T, e)= \gamma'_{\overline{0}}(T^0, e)+\sum\limits^{q}_{j=1} \gamma'_{0}(T^{j}, e^{j}).
\end{align*}
\end{proof}

By Theorem 3.1, we give algorithms as follows.

\begin{algorithm}[H]
\algsetup{linenosize=\small} \scriptsize
\caption{ Determine the value of $\gamma'_{1}(T, i')$.}
\begin{algorithmic}[1]
\REQUIRE an edge $i$ of a rooted tree $T$ which represent by its edge parent array $[1,2,3,\ldots,n]$.
\ENSURE $\gamma'_{1}(T,i')$
\STATE $i^{'}\leftarrow father (i)$;
\STATE $N_{c}(i')\leftarrow children (i')$;
\STATE $T^{0}\leftarrow ~the ~component ~containing ~i' ~of ~T-(N_{c}(i')) $;
\FOR{each $j\in N_{c}(i')$}
\STATE $T^{j}\leftarrow ~the ~component ~containing ~j ~of ~T-(N_{c}(i')+i'-j) $;
\ENDFOR
\FOR {each $j\in N_{c}(i')$}
    \STATE $\theta_{j}\leftarrow\min\{\gamma'_{1}(T^{j}, j),\gamma'_{0}(T^{j}, j),\gamma'_{\overline{1}}(T^{j},j),\gamma'_{\overline{0}}(T^{j}, j)\};$
    \STATE  $A_1\leftarrow\{j\in N_{c}(i')|\theta_{j}=\gamma'_{1}(T^{j}, j) \};$
    \STATE  $A_3\leftarrow\{j\in N_{c}(i')|\theta_{j}=\gamma'_{\overline{1}}(T^{j}, j) \};$
\ENDFOR
\IF {$A_1\cup A_3\neq\emptyset$ }
  \STATE  $\gamma'_{1}(T, i')\leftarrow\min\{\gamma'_{1}(T^0, i'),\gamma'_{\overline{1}}(T^0, i')\}+\sum\limits^{}_{j\in N_{c}(i')}\theta_{j}$
 \ELSE
    \STATE $\gamma'_{1}(T, i')\leftarrow\min\{\gamma'_{1}(T^0, i'), \gamma'_{\overline{1}}(T^0, i')+1\}+ \sum\limits^{}_{j\in N_{c}(i')}\theta_{j}$.
\ENDIF
 \end{algorithmic}
 \end{algorithm}

\begin{algorithm}
\algsetup{linenosize=\small} \scriptsize
\caption{ Determine the value of $\gamma'_{\overline{1}}(T, i')$.}
\begin{algorithmic}
\REQUIRE an edge $i$ of a rooted tree $T$ which represent by its edge parent array $[1,2,3,\ldots,n]$.
\ENSURE $\gamma'_{\bar{1}}(T,i')$
\STATE $i^{'}\leftarrow father (i)$;
\STATE $N_{c}(i')\leftarrow children (i')$;
\STATE $T^{0}\leftarrow ~the ~component ~containing ~i' ~of ~T-(N_{c}(i')) $;
\FOR{each $j\in N_{c}(i')$}
\STATE $T^{j}\leftarrow ~the ~component ~containing ~j ~of ~T-(N_{c}(i')+i'-j) $;
\ENDFOR
\STATE $\gamma'_{\overline{1}}(T, i')\leftarrow \gamma'_{\overline{1}}(T^0, i')+\sum\limits^{}_{j\in N_{c}(i')} \min \{\gamma'_{0}(T^{j}, j), \gamma'_{\overline{0}}(T^{j}, j)\}$
\end{algorithmic}
\end{algorithm}

\begin{algorithm}
\algsetup{linenosize=\small} \scriptsize
\caption{Determine the value of $\gamma'_{0}(T, i')$.}
\begin{algorithmic}
\REQUIRE an edge $i$ of a rooted tree $T$ which represent by its edge parent array $[1,2,3,\ldots,n]$.
\ENSURE $\gamma'_{0}(T,i')$;
\STATE $N_{c}(i')\leftarrow children (i')$;
\STATE $T^{0} \leftarrow ~the ~component ~containing ~i' ~of ~T-(N_{c}(i')) $;
\FOR{each $j\in N_{c}(i')$}
\STATE $T^{j}\leftarrow ~the ~component ~containing ~j ~of ~T-(N_{c}(i')+i'-j) $;
\ENDFOR
\FOR {each $j\in N_{c}(i')$}
    \STATE $\theta_{j}\leftarrow\min\{\gamma'_{1}(T^{j}, j),\gamma'_{0}(T^{j}, j), \gamma'_{\overline{1}}(T^{j},j),\gamma'_{\overline{0}}(T^{j}, j)\};$
     \STATE  $A_1 \leftarrow\{j\in N_{c}(i')|\theta_{j}=\gamma'_{1}(T^{j}, j) \};$
     \STATE  $A_2 \leftarrow\{j\in N_{c}(i')|\theta_{j}=\gamma'_{0}(T^{j}, j) \};$
    \STATE  $A_3 \leftarrow\{j\in N_{c}(i')|\theta_{j}=\gamma'_{\overline{1}}(T^{j}, j) \};$
    \STATE  $A_4 \leftarrow \{j\in N_{c}(i')|\theta_{j}=\gamma'_{\overline{0}}(T^{j}, j) \};$
 \ENDFOR

\IF{$A_1\neq\emptyset~or ~|A_3|\geqslant 2$}

\STATE  $\gamma'_{0}(T, i')\leftarrow \min\{\gamma'_{0}(T^0, i'), \gamma'_{\overline{0}}(T^0, i')\}+\sum\limits^{}_{j\in N_{c}(i')}\theta_{j}$

\ELSIF {$A_1=\emptyset~,and~|A_3|=1 ~or~ A_3=\emptyset, ~A_2\neq \emptyset, ~A_4\neq \emptyset$}
\STATE $\gamma'_{0}(T, i')\leftarrow\min\{\gamma'_{0}(T^0, i'), \gamma'_{\overline{0}}(T^0, i')\}+
 \sum\limits^{}_{j\in N_{c}(i')}\theta_{j}+1$

\ELSIF {$A_1=A_3=A_4=\emptyset$}

\STATE $\gamma'_{0}(T, i')\leftarrow\min\{\gamma'_{0}(T^0, i'),\gamma'_{\overline{0}}(T^0, i')+1\}+\sum\limits^{}_{j\in N_{c}(i')}\theta_{j}$
\ELSIF {$ A_1=A_2=A_3=\emptyset,~ and~\exists j\in A_4,~such~that~ \gamma'_{1}(T^{j},j)-\gamma'_{\overline{0}}(T^{j},j)==1$}
\STATE  $\gamma'_{0}(T, i')\leftarrow\min\{\gamma'_{0}(T^0, i'),\gamma'_{\overline{0}}(T^0, i')\}+\sum\limits^{}_{j\in N_{c}(i')}\theta_{j} +1$
\ELSIF {$A_1=A_2=A_3=\emptyset,~and~ \forall j\in A_4, \gamma'_{1}(T^{j},j)-\gamma'_{\overline{0}}(T^{j},j)==2$}
\STATE $\gamma'_{0}(T, i')\leftarrow\min\{\gamma'_{0}(T^0, i'),\gamma'_{\overline{0}}(T^0, i')\}+\sum\limits^{}_{j\in N_{c}(i')} +2$
\ENDIF
\end{algorithmic}
\end{algorithm}

\begin{algorithm}
\algsetup{linenosize=\small} \scriptsize
\caption{ Determine the value of $\gamma'_{\overline{0}}(T, i')$.}
\begin{algorithmic}
\REQUIRE an edge $i$ of an rooted tree $T$ which represent by its edge parent array $[1,2,3,\ldots,n]$.
\ENSURE $\gamma'_{1}(T,i')$
\STATE $i^{'}\leftarrow father (i)$;
\STATE $N_{c}(i')\leftarrow children (i')$;
\STATE $T^{0}\leftarrow ~the ~component ~containing ~i' ~of ~T-(N_{c}(i')) $;
\FOR{each $j\in N_{c}(i')$}
\STATE $T^{j}\leftarrow ~the ~component ~containing ~j ~of ~T-(N_{c}(i')+i'-j) $;
\ENDFOR

\STATE
$\gamma'_{\overline{0}}(T, i')\leftarrow \gamma'_{\overline{0}}(T^0, i')+\sum\limits^{m}_{j\in N_{c}(i')} \gamma'_{0}(T^{j}, j)$
\end{algorithmic}
\end{algorithm}

\begin{algorithm}
\algsetup{linenosize=\small} \scriptsize
\caption{Determine the total edge domination number of a tree. }
\begin{algorithmic}[1]
  \REQUIRE an edge rooted tree $T$ represent by its edge parent array $[1,2,3,\ldots,n]$.
  \ENSURE a minimum total edge domination number of $T$.
  \FOR {each $i\in[1,n]$ }
    \STATE $\gamma'_{1}(T,1)\leftarrow \infty$; $\gamma'_{0}(T,1) \leftarrow \infty$; $\gamma'_{\overline{1}}(T,1)\leftarrow1$; $\gamma'_{\overline{0}}(T,1)\leftarrow0$;
  \ENDFOR
  \FOR{each $i\in[1,n-1]$ }
   \STATE $i^{'}\leftarrow father (i)$;
    \STATE $N_{c}(i')\leftarrow children (i')$;
    \STATE $T^{0}\leftarrow ~the ~component ~containing ~i' ~of ~T-(N_{c}(i')) $;
    \FOR{each $j\in N_{c}(i')$}
\STATE $T^{j}\leftarrow ~the ~component ~containing ~j ~of ~T-(N_{c}(i')+i'-j) $;
\ENDFOR
    \FOR {each $j\in N_{c}(i')$}
    \STATE $\theta_{j}:=\min\{\gamma'_{1}(T^{j}, j),\gamma'_{0}(T^{j}, j), \gamma'_{\overline{1}}(T^{j},j),\gamma'_{\overline{0}}(T^{j}, j)\};$
     \STATE  $A_1:=\{j\in N_{c}(i')|\theta_{j}=\gamma'_{1}(T^{j}, j) \};$
     \STATE  $A_2:=\{j\in N_{c}(i')|\theta_{j}=\gamma'_{0}(T^{j}, j) \};$
    \STATE  $A_3:=\{j\in N_{c}(i')|\theta_{j}=\gamma'_{\overline{1}}(T^{j}, j) \};$
    \STATE  $A_4:=\{j\in N_{c}(i')|\theta_{j}=\gamma'_{\overline{0}}(T^{j}, j) \};$

    \STATE $\gamma'_{1}(T, i')$=Determine the value of $\gamma'_{1}(T, i')$.
    \STATE $\gamma'_{0}(T, i')$=Determine the value of $\gamma'_{0}(T, i')$.
    \STATE $\gamma'_{\bar{1}}(T, i')$=Determine the value of $\gamma'_{\overline{1}}(T, i')$.
    \STATE $\gamma'_{\bar{1}}(T, i')$=Determine the value of $\gamma'_{\overline{0}}(T, i')$.
   \ENDFOR

  \ENDFOR
  \RETURN $\gamma'(T) = min\{ \gamma'_{1}(T,n), \gamma'_{0}(T,n), \gamma'_{\bar{1}}(T,n), \gamma'_{\bar{0}}(T,n) \}$
 \end{algorithmic}
 \end{algorithm}
\begin{theorem}

Algorithm 5 produces the total edge domination number of a tree in linear-time.
\end{theorem}

\begin{proof}
It is easy to know that the running times of Algorithms 1, 2, 3 and 4 are constant times. Then Algorithm 5, needing to visit each father edge $e$ of $T$ once, and all of the statements within which can be executed in a constant time, so with an adequate data structure the algorithm works in linear-time.
\end{proof}

\section{Characterizing $(\gamma'_{t}=2\gamma')$-trees and $(\gamma'_{t}=\gamma')$-trees}

In this section we provide a constructive characterization of trees satisfying $ \gamma'_{t}(T) = 2\gamma'(T)$ and $\gamma'_{t}(T) = \gamma'(T)$, denoted by $(\gamma'_{t}=2\gamma')$-trees and $(\gamma'_{t}=\gamma')$-trees, respectively.

First, we begin with some properties of specific graphs used in this section.

\begin{example}\label{Ex:StaDou}
Let $T$ be a star or a double star. Then $\gamma^{'}(T)=1$ and $\gamma^{'}_{t}(T)=2$.
\end{example}

\begin{example}\label{Ex:Paths}
  If $T$ is a path with five vertices, then $\gamma^{'}(T)=\gamma^{'}_{t}(T)=2$. If $T$ is a path with six vertices, then $\gamma^{'}(T)=2$ and $\gamma^{'}_{t}(T)=3$.
\end{example}

\begin{theorem}\label{Th:NoLeafEdge}
Let $G$ be a connected graph of diameter $\geqslant 4$. Then there exists a minimum edge dominating set (resp. a minimum total edge dominating set) $D$ of $G$ such that  $D$ contains no leaf edges of $G$.
\end{theorem}

\begin{proof}
Suppose to the contrary that each minimum edge dominating set contains some leaf edges and $D$ is a minimum edge dominating set containing least leaf edges. Then for each leaf edge $e\in D$, $N(e)\cap D=\emptyset$, otherwise, $D-e$ is a smaller edge dominating set, a contradiction. Choose one non-leaf edge $e'$ of $N(e)$, then $D'=D-e+e'$ is a new minimum edge dominating set containing less leaf edges than $D$, a contradiction.  Similarly, we can prove the total version.
\end{proof}

\begin{corollary}\label{Co:diam4}
Let $T$ be a  tree with diameter 4. Then $\gamma^{'}_{t}(T)=\gamma^{'}(T)$.
\end{corollary}

\begin{proof}
The induced subgraph of all non-leaf edges in $T$ is a star $S_{1,k}$. In order to dominate all leaf edges, by Theorem \ref{Th:NoLeafEdge}, $E(S_{1,k})$ is a minimum edge dominating set and also a TED-set of $T$, so $\gamma^{'}_{t}(T)\leqslant\gamma^{'}(T)$, combined with $\gamma^{'}(T)\leqslant\gamma^{'}_{t}(T)$, we get $\gamma^{'}_{t}(T)=\gamma^{'}(T)$.
\end{proof}

\begin{corollary}\label{diam=5}
Let $T$ be a  tree with diameter 5. Then $\gamma^{'}_{t}(T)=\gamma^{'}(T)$ or
 $\gamma^{'}_{t}(T)=\gamma^{'}(T)+1$.
\end{corollary}

\begin{proof}
From the condition, the induced subgraph of all non-leaf edges in $T$ is exactly a double star, say $H$ and two adjacent center vertices $r, t$. Let $D$ be a minimum edge dominating set of $T$ containing no leaf edges by theorem \ref{Th:NoLeafEdge} and, $e$ an leaf edge in $H$, say $e=vr$ or $vt$. Since, in $T$, $v$ is incident with at least one leaf edge, $D$ contains $e$. Thus $(E(H)-rt)\subseteq D$. Combined that $D+rt$ induces an connected subgraph, exactly $H$, further $D+rt$ is a total edge dominating set of $T$, so $\gamma^{'}_{t}(T)=\gamma^{'}(T)$ or $\gamma^{'}_{t}(T)=\gamma^{'}(T)+1$.
\end{proof}

\subsection{$(\gamma'_{t}=2\gamma')$-trees}

In this subsection we provide a constructive characterization of trees $T$ satisfying $\gamma'_{t}(T)=2\gamma'(T)$. Note that a star or double star satisfies the condition above. In what follows we consider the trees satisfying the condition other than stars.

Our aim is to describe an inductive procedure of the tree $T$ with $\gamma'_{t}(T)=2\gamma'(T)$ by labelling.  For the initiated step, for any vertex $v$ of $P_4$, we give a label $C$ or $L$ to $v$, denoted by $l(v)$, defined as $l(v)=L$ if $v$ is a leaf of $P_4$, $l(v)=C$, otherwise. For convenience, we call an edge with both endpoints labelled $C$ as $C-C$ edge.

Let $\mathcal{T}$ be the family of labelled trees $T$ containing the labelled $P_{4}$ as the initiated labelled tree, constructed inductively by the two operations $\mathcal{O}_{1}$, $\mathcal{O}_{2}$ listed below (i.e., constructing a bigger labelled tree $T'$ from a smaller labelled tree $T$ in $\mathcal{T}$).\\

\noindent \textbf{Operation }$\mathcal{O}_{1}$:
Let $T \in \mathcal{T}$ and $v$ a vertex of $T$ with $l(v) = L$ such that: (1). each vertex labelled $C$ of distance 2 from $v$ is adjacent to a leaf vertex; (2). For any $C-C$ edge $wu$ of distance 1 from $v$, say $v$ is adjacent to $u$, either $u$ has a leaf other than $v$ or $N(w)-u$ are all leaves. Construct a bigger tree ${T'}$ in $\mathcal{T}$ from $T$ and a labelled $P_{4}$ by identifying $v$ and a leaf vertex of $P_{4}$, labelling the identified vertex as $L$ and keeping the labels of the other vertices unchanged, see Fig. \ref {Fig:1=2a}. \\

\noindent \textbf{Operation }$\mathcal{O}_{2}$:
 Let $T \in \mathcal{T}$ and $v$ a vertex of $T$ with $l(v) = C$. Construct a bigger tree ${T'}$ in $\mathcal{T}$ from $T$ by adding a new vertex $u$ adjacent to $v$, labelling $u$ as $L$, keeping the labels of the other vertices unchanged, see Fig. \ref{Fig:1=2b}.

\begin{figure*}[htbp]

\centering
\subfigure[\scriptsize{Operation $\mathcal{O}_{1}$;}]{
    \begin{minipage}[t]{0.4\linewidth}
        \centering
          \label{Fig:1=2a}
        \includegraphics[scale=0.6]{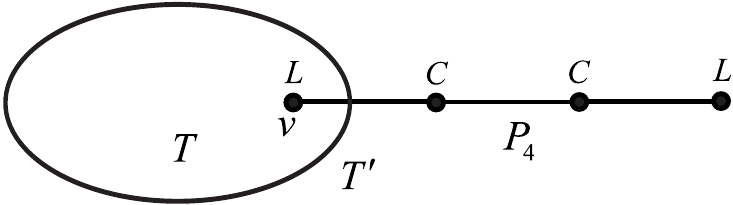}
        \vspace{0.46cm}

    \end{minipage}
}
\subfigure[\scriptsize{Operation $\mathcal{O}_{2}$.}]{
    \begin{minipage}[t]{0.4\linewidth}
        \centering
        \label{Fig:1=2b}
        \includegraphics[scale=0.6]{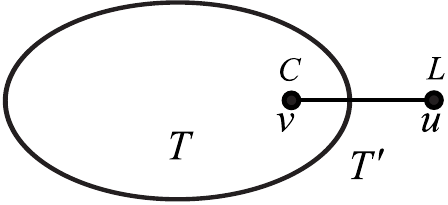}
        \vspace{0.46cm}
    \end{minipage}
}
\centering
\caption{Two operations.}
\vspace{-0.3cm}
\label{fig:1=2 operations}

\end{figure*}


From the two operations above, we can get the following simple observations.

\begin{observation}\label{obs:1}
Let $T\in\mathcal{T}$. Then
\begin{enumerate}
\item Each leaf vertex is labelled $L$ and each support vertex is labelled $C$.
\item Exactly one neighbor of each vertex labelled $C$ is labelled $C$, and the remaining neighbours are labelled $L$.
\item No two vertices labelled L are adjacent.
\item  If one endpoints of a $C-C$ edge has a non-leaf neighbor labelled $L$, then the other endpoint has one leaf neighbor.
\end{enumerate}
\end{observation}

\begin{lemma}\label{Le:CCEDS}
Let $T\in \mathcal{T}$ and $U$ the set of edges whose endpoints are labelled $C$ in $T$. Then $U$ is a $\gamma'(T)$-set.
\end{lemma}
\begin{proof}
By Observation \ref{obs:1} (2) and (3), we know that $U$ is an edge dominating set of $T$ and further each component of the induced subgraph $T[U]$ is $K_2$. By Observation \ref{obs:1} (4) and Theorem \ref{Th:NoLeafEdge}, the size of any edge dominating set is at least $|D|$. Thus,
 $U$ is a $\gamma'(T)$-set of $T$.
\end{proof}

\begin{lemma}\label{Le:TIs1=2}
Let $T\in\mathcal{T}$. Then $T$ is a $(\gamma'_{t}=2\gamma')$-tree.
\end{lemma}

\begin{proof}
We proceed by induction on the size $m$ of the edge set of a tree $T\in \mathcal{T}$. For the initial step, it is obvious that $\gamma'_{t}(P_4)=2\gamma'(P_4)$. For the inductive hypothesis, we assume that, for every $\overline{T}\in \mathcal{T}$ of edge size less than $m$, $\gamma'_{t}( \overline{T})=2\gamma'(\overline{T})$. Let $T\in \mathcal{T}$ with edge size $m$, and suppose $T$ is obtained from a tree $\overline{T}\in \mathcal{T}$ by one of two operations. We need to prove that $\gamma'_{t}( T)=2\gamma'(T)$. Next, we divide two cases to analyze according to which operation is used to construct the tree $T$ from $\overline{T}$.

\textbf{Case 1.} $T$ is obtained from $\overline{T}$ and a labelled $P_4=u_{1}u_{2}u_3u_4$ by Operation 1, i.e., identifying $u_1$ and $v (\in V(\overline{T}))$, denoted by $v$ the identifying vertex in $T$.

By Lemma \ref{obs:1}, we have $\gamma'(T)=\gamma'(\overline{T})+1$. Next, we just need to show $\gamma'_{t}(T)=\gamma'_{t}(\overline{T})+2$.

On the one hand, the union of a $\gamma'_{t}(\overline{T})$-set of $\overline{T}$ and $\{vu_{2}, u_{2}u_{3}\}$ is a TED-set of $T$, further $\gamma'_{t}(T)\leqslant\gamma'_{t}(\overline{T})+2$. On the other hand, it is sufficient to show that $\gamma'_t(\overline{T})+2\leq \gamma'_t(T)$. Without loss of generality, let $N_{\overline{T}}(v)=\{v_{1},\ldots, v_{r}\}$ for some positive integer $r$. For $1\leqslant i\leqslant r$, from the definition of Operation 1 and Observation \ref{obs:1} (3), $l_{\overline{T}}(v)=L$ and $l_{\overline{T}}(v_{i})=C$; by  Observation \ref{obs:1} (2), we denote by $w_{i}$ $(1\leqslant i\leqslant r)$ the unique vertex labelled $C$ adjacent to $v_i$ in $\overline{T}$; and by the choice of $v$ in the definition of Operation 1, $w_{i}$ has one leaf neighbor in $\overline{T}$.

By Theorem \ref{Th:NoLeafEdge}, we let $F_t$ be such a $\gamma'_{t}(T)$-set that $F_t$ contains no leaf edges. If the restriction $F_t|_{\overline{T}}$
 of $F_t$ on $\overline{T}$ is a TED-set of $\overline{T}$, then $\gamma'_t(\overline{T})+2\leq \gamma'_t(T)$. In what follows we assume that $F_t|_{\overline{T}}$ is not a TED-set of $\overline{T}$, then $|E_{\overline{T}}(v)\cap F_t|\leqslant 1$.

If $E_{\overline{T}}(v)\cap F_t=\emptyset$, then $F_t|_{\overline{T}}$ does not dominate some edge incident with $v$ in $\overline{T}$, say $vv_i$ for some integer $i$, further there is no leaf edge $e$ incident with $v_i$ in $T$, otherwise $F_t$ does not dominate $e$ in $T$. By the choice of $v$ in Operation 1,  all neighbors of $w_i$ other than $v_i$ are all leaves, a contradiction with the choice of $F_t$. If $E_{\overline{T}}(v)\cap F_t$ has a unique edge, say $vv_{i}$ for some $i$, then $w_iv_i\notin F_t$. Since $w_i$ has a leaf vertex by the choice of $v$ in Operation 1, there is one edge in $F_t$ incident with $w_i$. Therefore the restriction of $F_t-vv_i+v_iw_i$ on $\overline{T}$ is a TED-set of $\overline{T}$, further $\gamma'_t(\overline{T})+2\leqslant \gamma'_t(T)$.

\noindent\textbf{Case 2.} $T$ is obtained from $\overline{T}$ by adding a new vertex $u$ adjacent to $v$ labelled $C$ (i.e., Operation 2).

By Lemma \ref{Le:CCEDS}, we can easily get $\gamma'(T)=\gamma'(\overline{T})$. Then $\gamma'_{t}(T)\leqslant 2\gamma'(T)= 2\gamma'(\overline{T})=\gamma'_{t}(\overline{T})\leqslant\gamma'_{t}(T)$, and so $\gamma'_{t}(T)=2\gamma'(T)$.

Combined the two cases above, we have $\gamma'_{t}(T)=2\gamma'(T)$ for $T\in \mathcal{T}$.
\end{proof}

\begin{lemma}\label{Le:1=2CloNeiIntEmp}
Let $T$ be a tree with $\gamma'_{t}(T)=2\gamma'(T)$, $F$ a $\gamma'(T)$-set. Then $N[e]\cap N[e']=\emptyset$ for any distinct edges $e, e'\in F$.
\end{lemma}

\begin{proof}
By contradiction. Assume that there exist two edges $e, e'$ in $F$ such that $N[e]\cap N[e']\neq\emptyset$, say $e''\in N[e]\cap N[e']$. Now we construct a TED-set $S$ of $T$ from $F+ e''$: for any edge $f\in F-e- e'$, adding an edge adjacent to $f$ to $F+ e''$. Then $|S|\leqslant 2|F|-1=2\gamma'(T)-1$, a contradiction.
\end{proof}

\begin{corollary}\label{cor 1=2}
Let $T$ be a tree with $\gamma'_{t}(T)=2\gamma'(T)$,  $vu$ and $uw$ two adjacent edges in $T$. Then $v, w$ and $u$ can't all be support vertices.
\end{corollary}
\begin{proof}
 This follows directly from Lemma \ref{Le:1=2CloNeiIntEmp}.
\end{proof}

\begin{lemma}\label{Le:1=2IsT}
Let $T$ be a non-star tree with $\gamma'_{t}(T)=2\gamma'(T)$. Then $T\in\mathcal {T}$.
\end{lemma}

\begin{proof}
We proceed by induction on the edge size of a non-star tree $T$ with $\gamma'_{t}(T)=2\gamma'(T)$. For the initial step, if $T$ is a tree with $diam(T) =3$, then $T$ is a double star with $\gamma'(T)=1$ and $\gamma'_{t}(T)=2$, so we can obtain $T$ from a labelled $P_4$ by doing a series of Operation $\mathcal{O}_2$. By Corollaries \ref{Co:diam4} and \ref{diam=5}, if $T$ is a tree with $diam(T)=4$ or $5$, then $T$ does not satisfy $\gamma'_{t}(T)=2\gamma'(T)$. In what follows let $T$ be a tree of edge size $m$ and diameter at least 6 with $\gamma'_{t}(T)=2\gamma'(T)$. For the inductive hypothesis, we assume that every tree $\overline{T}$ of edge size less than $m$ with $\gamma'_{t}(\overline{T})=2\gamma'(\overline{T})$ is in $\mathcal {T}$.

If a support vertex $v$ has two leaf neighbor in $T$ with $\gamma'_{t}(T)=2\gamma'(T)$ and $w$ is one of leaf neighbors of $v$, then $v$ is still a support vertex in $\overline{T}=T-w$. Combined with Theorem \ref{Th:NoLeafEdge}, a minimum edge dominating set (resp. a minimum total edge set) of $\overline{T}$ containing no leaf edges is exactly a minimum edge edge dominating set (resp. a minimum total edge set) of $T$ containing no leaf edges. So $\gamma'(\overline{T})= \gamma'(T)$, $\gamma'_{t}(\overline{T})= \gamma'_{t}(T)$. Therefore, $2\gamma'(\overline{T})= 2\gamma'(T)=\gamma'_{t}(T)=\gamma'_{t}(\overline{T})$. By the inductive hypothesis, $\overline{T}\in \mathcal{T}$ with a labeling. By Observation \ref{obs:1} (1), the support vertex $v$ is labelled $C$ in $\overline{T}$. Thus we can obtain the tree $T$ by applying Operation $\mathcal{O}_{2}$ to $\overline{T}$.

Let $P$ be a longest path in $T$, say $P=v_{0}v_{1}\ldots v_{t}$ for some $t$ ($t \geqslant 6$) and denoted by $e_{i}=v_{i}v_{i+1}$.  If $v_2$ has a leaf neighbor, say $v'_1$, let $\overline{T}=T-v'_1$. By Theorem \ref{Th:NoLeafEdge}, $\overline{T}$ has a $\gamma'(\overline{T})$-set (resp. a $\gamma'_t(\overline{T})$-set ) containing $e_1$, which is still a $\gamma'(T)$-set (resp. a $\gamma'_t(T)$-set), so
$\gamma'_t(\overline{T})=\gamma'_t(T)=2\gamma'(T)=2\gamma'(\overline{T}).$ By the inductive hypothesis, $\overline{T}\in \mathcal{T}$ with a labelling. By Observation \ref{obs:1} (1) and (2), the vertices $v_1$ and $v_2$ are labelled $C$ in $\overline{T}$. Thus we can obtain the tree $T$ by applying Operation $\mathcal{O}_{2}$ to $\overline{T}$.

In what follows we assume that each support vertex of $T$ has exactly one leaf neighbor and $v_2$ is not a support vertex. Let $F$ be a $\gamma'(T)$-set of $T$ containing non-leaf edges by Theorem \ref{Th:NoLeafEdge}, thus $e_1\in F$. For convenience, we root $T$ at the vertex $v_{t}$.

\begin{claim}\label{Cl:1=2IsTv3Chi}
For every child $v$ of $v_3$, the subtree of $T-v_3$ containing $v$ is exactly $P_3$.
\end{claim}

By contradiction. If $v$ is a leaf, then there exists an edge incident with $v_3$ in $F$, say $e$. Note that $e_1\in F$. But $N[e]\cap N[e_1]\neq \emptyset$, a contradiction with Lemma \ref{Le:1=2CloNeiIntEmp}. If $v$ has only leaf children, then $vv_3\in F$. Similarly, we can obtain a contradiction because $N[vv_3]\cap N[e_1]\neq \emptyset$. If $v$ has at least two support children, then $|E(v)\cap F|\geqslant 2$, a contradiction. So $v$ has exactly one support child, combined with the same role of $v$ as $v_2$ in the choice of $P$ and the assumption, we obtain the claim.

\begin{claim}\label{Cl:1=2IsTv4Chi}
For  a child $v'_3$ of $v_4$, the length of a longest path starting at $v'_3$  in the subtree $T-v_4$ containing $v'_3$ is not 2.
\end{claim}

Assume to the contrary that there exists  one child $v'_3$ of $v_4$ such that  the length of a longest path $P$ starting at $v'_3$  in the subtree $T-v_4$ containing $v'_3$ is 2, say $P=v'_3v'_2v'_1$. Obviously, $v'_3\neq v_3$ and $v'_3v'_2\in F$.  Combined with Lemma \ref{Le:1=2CloNeiIntEmp} and $e_1\in F$, $E(v_4)\cap F=\emptyset$, then $e_3$ is not dominated by $F$, a contradiction.

 \begin{claim}\label{Cl:1=2IsTv5Chi}
If there exists a child $v'_3$ of $v_4$ such that  the subtree of $T-v_4$ containing $v'_3$ is $P_2$, then $v_5$ has no leaf child.
 \end{claim}

Similar to the analysis of Claims \ref{Cl:1=2IsTv3Chi} and \ref{Cl:1=2IsTv4Chi}, we can show it by contradiction.

\begin{claim}\label{Cl:1=2IsTv6Chi}
If $diam(T)>6$ and there exist no children $v'_3$ of $v_4$ such that  the subtree of $T-v_4$ containing $v'_3$ is $P_2$, then $v_4$ and $v_5$ are both support vertices.
 \end{claim}

Suppose to the contrary. Since there is no subtree of $T\setminus \{v_4\}$ containing $v'_3$ isomorphic to $P_2$.

\begin{figure*}[htbp]
\centering
\subfigure[A total edge dominating set $F_t$ of $T$;]{
    \begin{minipage}[t]{0.4\linewidth}
        \centering
          \label{Fig:1=2a}
        \includegraphics[scale=0.45]{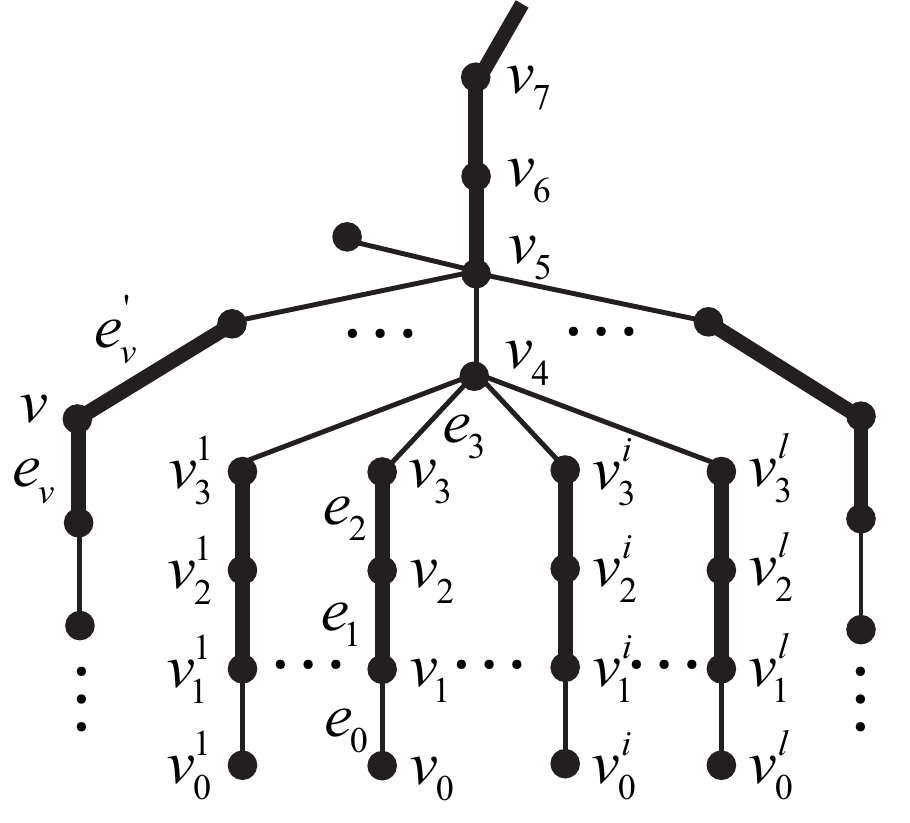}
        \vspace{0.12cm}
           \end{minipage}
}
\subfigure[a minimum  total edge dominating set $F'_t$ of $T$.]{
    \begin{minipage}[t]{0.4\linewidth}
        \centering
        \label{Fig:1=2b}
        \includegraphics[scale=0.45]{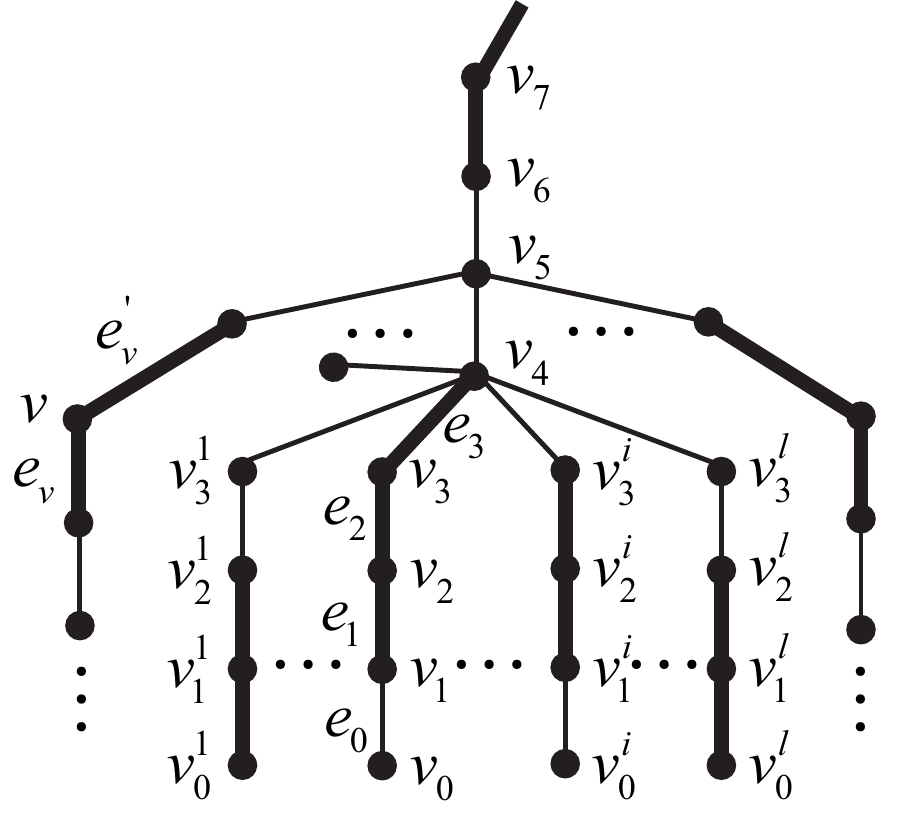}
        \vspace{0.12cm}
    \end{minipage}
}
\centering
\caption{Illstration for Claim \ref{Cl:1=2IsTv6Chi} in Lemma \ref{Le:1=2IsT}.}
\vspace{-0.3cm}
\label{Fig:1=2IsTv_6Chi}
\end{figure*}

Let $\{v_3^1, v_3^2, \ldots, v_3^s\}$ be the set of non-leaf children of $v_4$ for some positive integer $s$. For any $1\leqslant i\leqslant l$,  combined the assumption and Claim \ref{Cl:1=2IsTv4Chi}, the length of a longest path starting at $v_3^i$ in the subtree $T-v_4$ containing $v_3^i$ is 3. By the symmetry of $v_3$ and $v_3^i$ and Claim \ref{Cl:1=2IsTv3Chi}, the subtree of $T-v_4$ containing $v_3^i$ is exactly $P_4$, say $v_3^iv_2^iv_1^iv_0^i$.  By the choice of $F$ and Lemma \ref{Le:1=2CloNeiIntEmp}, for any $1\leqslant i\leqslant s$, $v_2^iv_1^i\in F$ and $E(v_3^i)\cap F=\emptyset$. So $\{e_{4}\}\subseteq F$.

If $v_4$ is not support,  let $e'_7$ be the unique edge in  $(E_{T}(v_7)- e_6)\cap F\neq \emptyset$ by $\{e_{4}\}\subseteq F$ and Lemma \ref{Le:1=2CloNeiIntEmp}.  Now we construct a TED-set $F_t$ of $T$ from ${F}_{0}=F-e_{4}+e_5$ by first adding the common neighbor edge $e_6$ of $e_5$ and $e'_7$ in $F_0$, second, for any $1\leqslant i\leqslant l$, adding $v_2^iv_3^i$ into $F_0$, and adding a neighbor edge of each edge in $F_0- \{e_5, e'_7\}+\{v_1^1v_2^1, v_1^2v_2^2, \cdots, v_1^sv_2^s \}$ (see Fig. \ref{Fig:1=2IsTv_6Chi}(a)). It is obvious that $F_t$ is a TED-set of $T$ and $|F_t|\leqslant 2|F|-1$, a contradiction.

If $v_5$ is not support, let $A$ be the set of vertices of distance 2 from $v_5$ in the subtree of $T-e_4$ containing $v_5$. For $v\in A$, $|F\cap E(v)|=1$, say $e_v$, by $e_4\in F$ and Lemma \ref{Le:1=2CloNeiIntEmp}, and denoted by $e'_v$ the unique edge of $E(v)$ of distance 1 from $v_5$. Note that $e'_v\neq e_v$ because $v_5$ is of distance 2 from $e_v$. Now we can construct a TED-set $F'_t$ of $T$ from ${F}'_{0}=F-e_{4}+e_3$ by first adding $e_2$ and the set $\{e'_v|v\in A\}$ into $F'_0$, second adding a neighbor edge of each edge in $F'_0-\{e_v| v\in A\}- \{e_1, e_3\}$ (see Fig. \ref{Fig:1=2IsTv_6Chi}(b)). Note that $\{e_1, e_2, e_3\}\in F'_t$. It is obvious that $F'_t$ is a TED-set of $T$ and $|F'_t|\leqslant 2|F|-1$, a contradiction. So we prove Claim \ref{Cl:1=2IsTv6Chi}.

By Claim \ref{Cl:1=2IsTv3Chi} and the assumption that each support vertex of $T$ has exactly one leaf neighbor and $v_2$ is not a support vertex, $d(v_1)=d(v_2)=2$, thus the subgraph induced by $\{v_0, v_1, v_2, v_3\}$ is $P_4$.  Let $\overline{T}=T-\{v_{0},v_{1},v_{2}\}$.

\begin{claim}\label{Cl:1=2IsTCut1=2}
$\gamma'_{t}(\overline{T}) = 2\gamma'(\overline{T})$.
\end{claim}

Combined with Lemma \ref{Le:1=2CloNeiIntEmp} and  $e_1\in F$, we have $E(v_3)\cap F= \emptyset$, thus the restriction of $F$ on $\overline{T}$ is an ED-set of $\overline{T}$, further $\gamma'(\overline{T})\leqslant\gamma'(T)-1$.  Combined with the obvious inequality: $\gamma'_{t}(T)\leqslant\gamma'_{t}(\overline{T})+2$, we have $2\gamma'(\overline{T})\leqslant2(\gamma'(T)-1) = 2\gamma'(T)-2 = \gamma'_{t}(T)-2\leqslant\gamma'_{t}(\overline{T})\leqslant2\gamma'(\overline{T})$.
Consequently we must have equality throughout this inequality chain. Particularly, we have $\gamma'_{t}(\overline{T})=2\gamma'(\overline{T})$.

By Claim \ref{Cl:1=2IsTCut1=2} and the inductive hypothesis, $\overline{T}\in \mathcal{T}$ with a labelling. In what follows we show that $T$ is obtained from $\overline{T}$ by Operation $\mathcal{O}_1$ (the identifying vertex is $v_3$, the role of $v$). By Claim \ref{Cl:1=2IsTv3Chi} and Observation \ref{obs:1} (1), (2), (3), we have $l(v_3)=L$, $l(v_4)=C$. In the case $l(v_5)=C$, if $diam(T)=6$, then all neighbors of $v_5$ other than $v_4$ are all leaves; if $diam(T)>6$, then each child of $v_4$ is labelled $L$ and by Claim \ref{Cl:1=2IsTv6Chi}, $v_4$ and $v_5$ have both one leaf neighbor. In the other case $l(v_5)=L$, there is one child $v'_3$ of $v_4$ labelling C. Combined with Corollary \ref{obs:1} and Claim \ref{Cl:1=2IsTv4Chi}, $v'_{3}$ has only leaf children. For the other C-C edges $wu$ of distance 1 from $v_3$, say $v_3$ is adjacent to $u$, i.e., $u$ is the child of $v_3$, by Claim \ref{Cl:1=2IsTv3Chi}, $N(w)-u$ are all leaves. Combined all cases above, it is obvious that the $C-C$ edge incident with $v_4$ of distance 1 from $v_3$ satisfies the condition in Operation $\mathcal{O}_1$. Therefore we can apply Operation $\mathcal{O}_1$ from $\overline{T}$ to obtain the tree $T$, further, ${T}\in \mathcal{T}$.
\end{proof}

As an immediate consequence of Lemmas \ref{Le:TIs1=2} and \ref{Le:1=2IsT}, we have

\begin{theorem}\label{Th:1=2NecSuf}
A non-star tree is a $(\gamma'_{t}=2\gamma')$-tree if and only if $T\in \mathcal{T}$.
 \end{theorem}

\subsection{$(\gamma'_{t}=\gamma')$-trees}

In this subsection we provide a constructive characterization of $(\gamma'_{t}=\gamma')$-trees $T$, i.e., a tree satisfying $\gamma'_{t}(T)=\gamma'(T)$. We use edge labelling to describe a procedure of constructing $T$ recursively, which is different from the vertex labelling in the previous subsection. By Example \ref{Ex:StaDou} and Corollary \ref{Co:diam4}, for the initial step, let $T$ be a tree with $diam(T)=4$, in which each edge is either a leaf edge or a support edge, we label support edges in $T$ with $S$, leaf edges adjacent to at least two non-leaf-edges with $L_{2}$, other leaf edges with $L_{1}$.

Let $\mathcal{T}_{t}$ be the family of edge-labelled trees $T$ that contains edge-labelled trees with diameter 4 and is under the five operations $\mathcal{O}_{1}$, $\mathcal{O}_{2}$, $\mathcal{O}_{3}$, $\mathcal{O}_{4}$, $\mathcal{O}_{5}$ listed below: constructing a bigger tree from a smaller tree in $\mathcal{T}_{t}$. For convenience, we call an edge labelled $S$ (resp. $L_1, L_2$) in $T\in \mathcal{T}_t$ an $S$ (resp. $L_1, L_2$)-edge, and denote by $D(T)$ the set of $S$-edges. First, according to the label of the associated edges of the vertex $v$ in an edge-labelled tree $T\in \mathcal{T}_t$, we partition the vertex set of $T$ into the following four subsets $A_{1}, A_{2}, B$ and $C$ listed below:
\begin{align*}
A_{1}:=&\{v| \text{ Only one } S-\text{edge in } E(v) \};\\
A_{2}:=&\{v| \text{ At least two } S-\text{edges in } E(v)\};\\
B:=&\{v|\text{ All edge in } E(v) \text{ are } L_{2}-\text{edges}\};\\
C :=& V-A_{1}-A_{2}-B.
\end{align*}

\begin{figure*}[h]
\subfigure[ \scriptsize{$v\in A_1$: exactly one $S$-edge.}]{
    \begin{minipage}[t]{0.22\linewidth}
    \centering
  \includegraphics[scale=0.45]{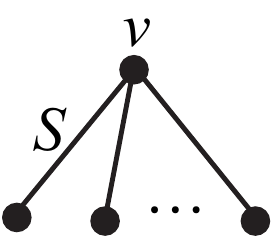}
        \vspace{0.02cm}
        \end{minipage}
}
\subfigure[\scriptsize{$v\in A_2$: at least two $S$-edges.}]{
  \begin{minipage}[t]{0.22\linewidth}
  \centering
    \includegraphics[scale=0.45]{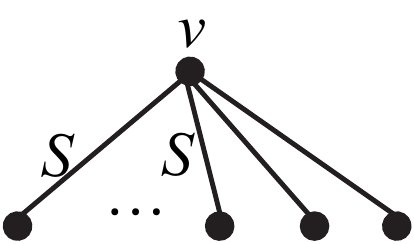}
    \vspace{0.02cm}
  \end{minipage}
}
\subfigure[\scriptsize{$v\in B$: all edges in $E(v)$ are $L_2$-edges.}]{
  \begin{minipage}[t]{0.22\linewidth}
  \centering
    \includegraphics[scale=0.45]{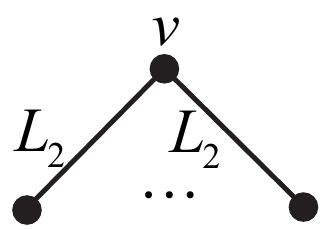}
    \vspace{0.02cm}
  \end{minipage}
}
\subfigure[\scriptsize{$v\in C$: at least one $L_1$-edge but no $S$-edges in $E(v)$.}]{
  \begin{minipage}[t]{0.22\linewidth}
  \centering
    \includegraphics[scale=0.45]{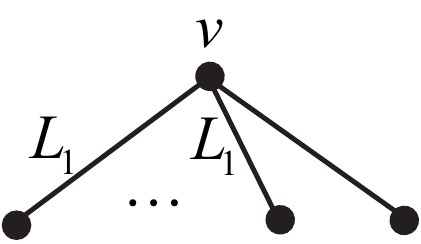}
    \vspace{0.02cm}
  \end{minipage}
}
\centering
\caption{Vertex partition of $T\in \mathcal{T}$.}
\vspace{-0.3cm}
\label{Fig:1=1VerLab}
\end{figure*}

Now, we list the five operations $\mathcal{O}_{1}$, $\mathcal{O}_{2}$, $\mathcal{O}_{3}$, $\mathcal{O}_{4}$, $\mathcal{O}_{5}$:

\noindent\textbf{Operation} \textbf{$\mathcal{O}_{1}$}:
Let $T\in \mathcal{T}_{t}$, $v$ a vertex of $T$ belonging to $A_{1}\cup A_{2}$. Construct a bigger tree $T'$ in $\mathcal{T}_{t}$ from $T$ by adding a new vertex $u$ adjacent to $v$. If $v\in A_{1}$, then label $vu$ as $L_1$; (by definition,  $u$ is in $C$,  $A_1, A_2, B$ are unchanged;) if $v\in A_{2}$, then label $vu$ as $L_2$ (note that $u\in B$ and  $A_1, A_2, C$ are unchanged), see Fig. \ref{fig:1=1a}. \\

\noindent\textbf{Operation} \textbf{$\mathcal{O}_{2}$}:
Let $T \in \mathcal{T}_{t}$, $v$ a vertex of $T$ belonging to $A_{2}$.  Construct a bigger tree $T'$ in $\mathcal{T}_{t}$ from $T$ by adding two new adjacent vertices  $u_{1}, u_{2}$, connecting $v$ and $u_1$ and labelling $vu_{1}$ as $S$ and $u_{1}u_{2}$ as $L_{1}$ (obviously, $u_1\in A_1$ and  $u_2\in C$), see Fig. \ref{fig:1=1b}.\\

\noindent\textbf{Operation} \textbf{$\mathcal{O}_{3}$}:
Let $T \in \mathcal{T}_{t}$, $v\notin A_1$ a vertex of $T$ satisfying, in the case $v\in C$, that each $L_1$-edge in $E(v)$ is either adjacent to one leaf edge or contained in a $P_4=vwxy$, whose edges are labelled as $L_1, L_1, L_2$ consecutively and all edges in $E(x)$ are $L_2$-edges except $wx$. Construct a bigger tree $T'$ in $\mathcal{T}_{t}$ from $T$ by adding a new path $u_{1}u_{2}u_{3}u_{4}u_{5}$ to join $v$ and $u_2$, and labelling $u_{2}u_{3}$, $u_{3}u_{4}$ as $S$, $vu_{2}$, $u_{1}u_{2}$, $u_{4}u_{5}$ as $L_{1}$, see Fig. \ref{fig:1=1c}. (From the definition, $u_2, u_4 \in A_1$, $u_3\in A_2$, $u_1, u_5\in C$ and if $v\in B$, then $v$ is moved from $B$ to $C$.)\\

\noindent\textbf{Operation} \textbf{$\mathcal{O}_{4}$}:
Let $T \in \mathcal{T}_{t}$, $v\in B$ a vertex of $T$. Construct a bigger tree $T'$ in $\mathcal{T}_{t}$ from $T$ by adding a new path $u_{1}u_{2}u_{3}u_{4}$ to join $v$ and $u_1$, and labelling $vu_{1}, u_{3}u_{4}$ as $L_{1}$, $u_{1}u_{2}, u_{2}u_{3}$ as $S$, see Fig. \ref{fig:1=1d}. (Similarly, $u_1, u_3\in A_1$, $u_2\in A_2$,  $u_4\in C$, and $v$ is moved from $B$ to $C$.)\\

\noindent\textbf{Operation} \textbf{$\mathcal{O}_{5}$}:
Let $T \in \mathcal{T}_{t}$, $v$ a vertex of $T$. Construct a bigger tree $T'$ in $\mathcal{T}_{t}$ from $T$ by adding a new path $u_{1}u_{2}u_{3}u_{4}u_{5}$ to join $v$ and $u_{3}$, and labelling $vu_3$ as $L_{2}$,  $u_{1}u_{2}, u_{4}u_{5}$ as $L_{1}$, $u_{2}u_{3}, u_{3}u_{4}$ as $S$, see Fig. \ref{fig:1=1e}. (From the definition, $u_2, u_4\in A_1$, $u_1, u_5\in C$, $u_3\in A_2$ and if $v\in B$, then $v$ is moved from $B$ to $C$.)
\\

\begin{figure*}[h]
\subfigure[\scriptsize{Operation $\mathcal{O}_{1}$:  the edge $uv$ is labelled as $L_2$ if $v\in A_2$ and is labelled as $L_1$
        if $v\in A_1$.}]{
    \begin{minipage}[t]{0.3\linewidth}
    \centering
    \label{fig:1=1a}
        \includegraphics[scale=0.45]{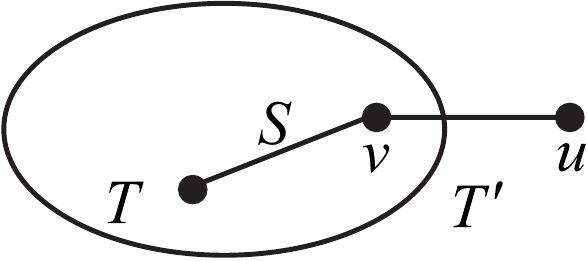}
        \end{minipage}
}
\subfigure[\scriptsize{Operation $\mathcal{O}_{2}$.}]{
  \begin{minipage}[t]{0.3\linewidth}
  \centering
  \label{fig:1=1b}
    \includegraphics[scale=0.45]{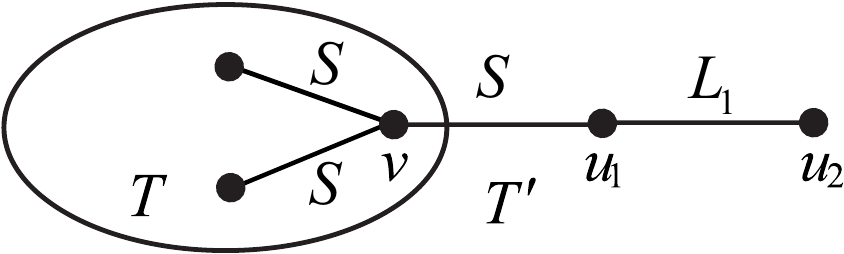}
  \end{minipage}
}
 \subfigure[\scriptsize{Operation $\mathcal{O}_{3}$.}]{
  \begin{minipage}[t]{0.3\linewidth}
  \centering
  \label{fig:1=1c}
    \includegraphics[scale=0.45]{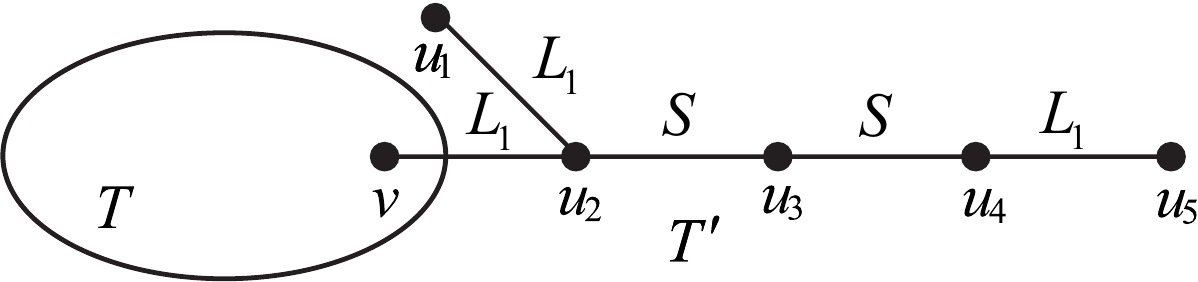}
  \end{minipage}
}

\subfigure[\scriptsize{Operation $\mathcal{O}_{4}$.}]{
  \begin{minipage}[t]{0.46\linewidth}
  \centering
  \label{fig:1=1d}
    \includegraphics[scale=0.45]{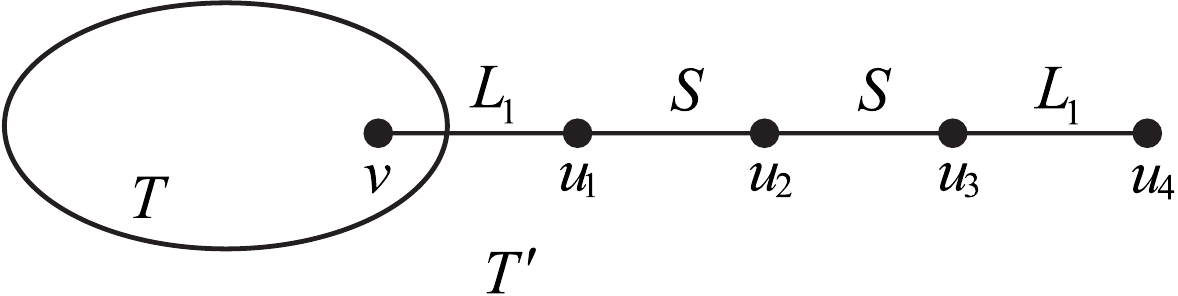}
  \end{minipage}
}
\subfigure[\scriptsize{Operation $\mathcal{O}_{5}$.}]{
  \begin{minipage}[t]{0.46\linewidth}
  \centering
  \label{fig:1=1e}
    \includegraphics[scale=0.45]{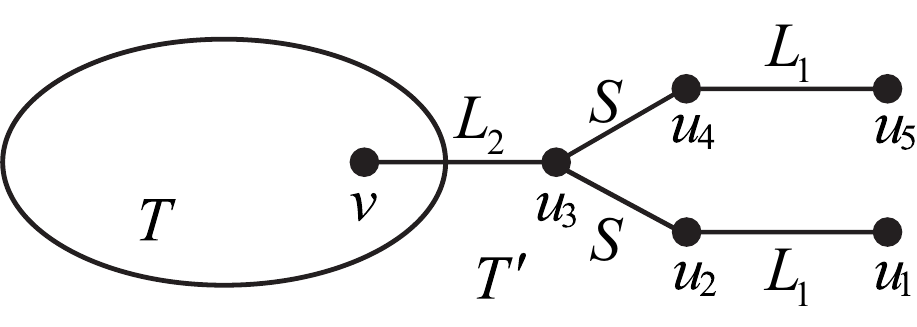}
  \end{minipage}
}
\centering
\caption{Five operations.}
\vspace{-0.3cm}
\label{fig:1=1}
\end{figure*}

From the five operations above, we can get the simple observations as follows.

\begin{observation} \label{ob2}

Let $T\in\mathcal{T}_{t}$.
\begin{enumerate}
\item \label{Ob:EndpOfL1}One endpoint of an $L_1$-edge is incident with exactly one $S$-edge, the other endpoint is incident with either non $S$-edges or at least two $S$-edges.
\item\label{Ob:L2NbSedge} An $L_2$-edge is adjacent to at least two $S$-edges.
\item\label{Ob:LEisL1orL2} A leaf edge is labelled $L_1$ or $L_2$. Furthermore, a leaf edge adjacent to exactly one non-leaf edge $e$ is labelled $L_1$ and $e$ is labelled as $S$.
\item\label{Ob:SedgeisTED} Each edge in $T$ is adjacent to at least one $S$-edge, and each component of the induced subgraph $T[D(T)]$ is a nontrivial star. Further, $D(T)$ is a total edge dominating set of $T$.
\end{enumerate}
\end{observation}

\begin{lemma}\label{Le:TIs1=1}
Let $T\in\mathcal{T}_{t}$. Then $D(T)$ is a $\gamma'_{t}(T)$-set and $T$ is a $(\gamma'_{t}=\gamma')$-tree.
\end{lemma}

\begin{proof}
Let $T\in \mathcal{T}_t$, we first prove that $D(T)$ (simply, $D$) is a $\gamma'_{t}(T)$-set. By Observation \ref{ob2} (\ref{Ob:SedgeisTED}), $D$ is a TED-set of $T$. It is sufficient to find a set $L$ of $L_1$-edges of size $|D|$ such that 
 each edge in $L$ has exactly one neighbor $S$-edge. 
In order to prove that there is such an edge set of each tree $T$ in $\mathcal{T}_{t}$, we proceed by induction on the size $m$ of the edge set of $T$. For the initial step, the leaves adjacent to exactly one non-leaf edge
of $T$ with diameter 4 construct the required set $L$.  For the inductive step, we assume each tree $\overline{T}$ of size less than $m$ in $\mathcal{T}_t$ has  a set $\overline {L}$ of $L_1$-edges such that each edge in $\overline{L}$ has exactly one neighbor $S$-edge. Now we divide five cases as follows:

 \noindent\textbf{Case 1.} $T$ is obtained by applying Operation $\mathcal{O}_{1}$ from $\overline{T}$ and a vertex $u$.

In this case, $D(T)=D(\overline{T})$,  and let  $L=\overline{L}$, which is the desired set for $T$.

\noindent\textbf{Case 2.} $T$ is obtained by applying Operation $\mathcal{O}_{2}$ from $\overline{T}$ and an edge $u_1u_2$ in which a vertex $v$ in $\overline{T}$ is adjacent to $u_1$.

In this case, $D(T)$ is one more $S$-edge than $D(\overline{T})$. By Observation \ref{ob2} (\ref{Ob:EndpOfL1}),  there is no $L_1$-edges in $\overline{L}$ incident with $v$. So $\overline {L}\cup \{u_1u_2 \}$ is a desired set for $T$.

\noindent\textbf{Case 3.} $T$ is obtained by applying Operation $\mathcal{O}_{3}$ from $\overline{T}$ and a path $u_1u_2u_3u_4u_5$.

In this case, $D(T)$ is two more edges than $D(\overline{T})$. If $v\in A_2\cup B$, by the definitions of $A_2, B$ and Observation \ref{ob2} (\ref{Ob:EndpOfL1}), there are no $L_1$-edges incident with $v$ in $\overline{T}$. So $\overline {L}\cup \{u_1u_2, u_4u_5\}$  is a desired set for $T$.

When $v\in C$, if there is no $L_1$-edge incident with $v$ in $\overline{L}$, then $\overline {L}\cup \{u_1u_2, u_4u_5\}$ is a desired set for $T$. Otherwise, let $e'=vw$ be the $L_1$-edge in $\overline {L}$, from the definition of Operation $\mathcal{O}_3$, there is one leaf edge $e''$ incident with $w$ or there exists a $P_4=vwxy$ in $\overline{T}$, whose edges are labelled as $L_1, L_1, L_2$ consecutively and all edges in $E(x)$ are $L_2$-edges except $wx$, then $(\overline {L}-e')\cup \{u_1u_2, u_4u_5, e''\}$ or $(\overline {L}-e')\cup \{u_1u_2, u_4u_5, wx\}$ is a desired set for $T$.

Therefore, we can always find a desired set for $T$ in this case.

\noindent\textbf{Case 4.} $T$ is obtained by applying Operation $\mathcal{O}_{4}$ from $\overline{T}$ and a path $u_1u_2u_3u_4$.

In this case, $D(T)$ is two more edges than $D(\overline{T})$.
If there is no $L_1$-edge in $\overline {L}$ adjacent to some edge in $E(v)$, then $\overline {L}\cup \{vu_1, u_3u_4, \} $ is a desired set for $T$. Otherwise, let $wx$ be the $L_1$-edge in $\overline {L}$ adjacent to some edge in $E(v)$. By Observation \ref{ob2} (\ref{Ob:EndpOfL1}),(\ref{Ob:L2NbSedge}), without loss of generality, assume $x\in A_1$, then there is an $L_1$-edge $xy$ in $E(x)$ such that $y$ is either a leaf vertex or only incident with $L_2$-edges except $xy$.
So $(\overline {L}-wx)\cup \{vu_1, u_3u_4, yx\}$ is a desired set for $T$.

Hence, we can always find a desired set for $T$ in this case.

\noindent\textbf{Case 5.} $T$ is obtained by applying Operation $\mathcal{O}_{5}$ from $\overline{T}$ and a path $u_1u_2u_3u_4u_5$.

In this case, $D(T)$ is two more edges than $D(\overline{T})$. So $\overline {L}\cup \{u_1u_2, u_4u_5\}$ is a desired  edge set  for $T$.

Combined the five cases above, for $T \in \mathcal{T}_{t}$, we can always find an edge set $L$ collecting $L_1$-edge such that each edge in $L$ has exactly one neighbor $S$-edge. Since the edges in $L$ need at least $|L|$ edges to dominate, $\gamma'(T)\geqslant |L|=|D|$. Hence, $D$ is a $\gamma'_{t}(T)$-set and $T$ is a $(\gamma'_{t}=\gamma')$-tree
\end{proof}

\begin{lemma}\label{Le:1=1NonTriSta}
Let $T$ be a $(\gamma'_{t}=\gamma')$-tree, $F_{t}$ a $\gamma'_{t}(T)$-set. Then any component of the induced subgraph $T[F_{t}]$ is nontrivial star.
\end{lemma}

\begin{proof}
By contradiction. If there is a $P=v_1v_2v_3v_4$ in $T[F_{t}]$, then the edges which are dominated by $v_2v_3$ are also dominated by $v_1v_2$ or $v_3v_3$. So $F_{t}-v_2v_3$ is an edge dominating set with cardinality $|F_{t}|-1$, a contradiction. Hence every component of $T[F_{t}]$ is a nontrivial star.
\end{proof}

\begin{lemma}\label{Le:1=1IsT}
Let $T$ be a $(\gamma'_{t}=\gamma')$-tree. Then $T\in\mathcal{T}_{t}$.
\end{lemma}

\begin{proof}
We proceed by induction on the edge size of a nontrivial tree $T $ satisfying $\gamma'_{t}(T)=\gamma'(T)$. For the initial step, by Corollary \ref{Co:diam4}, a tree $T$ with diameter 4
satisfies $\gamma'_{t}(T)=\gamma'(T)$ and is in $\mathcal{T}_t$. For the inductive hypothesis, we assume that every tree $\overline{T}$ with $\gamma'_{t}(\overline{T})=\gamma'(\overline{T})$ has edge size less than $m$ and $diam(\overline{T})\geqslant 5$, there exists an edge label such that $\overline{T}\in\mathcal {T}$.

If a support vertex $v$ of $T$ has at least two leaf neighbors, say $u$ and $w$ two of them, then $v$ is still a support vertex in $\overline{T}=T-w$. By Theorem \ref{Th:NoLeafEdge}, any minimum edge dominating set of $\overline{T}$ containing no leaf edges is still an edge dominating set of $T$. So $\gamma'(T)=\gamma'_{t}(T)=\gamma'(\overline{T})\leqslant \gamma'_{t}(\overline{T})\leqslant \gamma'_{t}(T) $ and $\gamma'_{t}(\overline{T})=\gamma'(\overline{T})$. Hence, by the inductive hypothesis, $\overline{T}\in\mathcal{T}_{t}$. By Observation \ref{ob2} (\ref{Ob:EndpOfL1}), (\ref{Ob:L2NbSedge}), (\ref{Ob:LEisL1orL2}), $uv$ is an $L_1$- or $L_2$-edge and $v\in A_2\cup A_1$  in $\overline{T}$. We can obtain the tree $T$ by applying Operation $\mathcal{O}_{1}$ from $\overline{T}$ and a new vertex $w$, so $T\in\mathcal{T}_{t}$. We may assume that each support vertex of $(\gamma'_t=\gamma')$-tree $T$ of edge size $m$ has exactly one leaf neighbor, denoted by Assumption 1.

If a support vertex $v$ of $T$, say $w$ is a leaf neighbor of $v$, has a support neighbor with degree 2, then let $\overline{T}=T -w$. Similar to the discuss as above, $\gamma'_{t}(\overline{T})=\gamma'(\overline{T})$ and by the inductive hypothesis, $\overline{T}\in \mathcal{T}_t$. By Observation \ref{ob2} (\ref{Ob:LEisL1orL2}), $uv$ is an $S$-edge in $\overline{T}$, combined with Observation \ref{ob2} (\ref{Ob:SedgeisTED}), $v\in A_2$.  We can obtain the tree $T$ by applying Operation $\mathcal{O}_{1}$ from $\overline{T}$ and a new vertex $w$, so $T\in\mathcal{T}_{t}$. We may assume that there is no support vertex which has a support neighbor of degree 2, denoted by Assumption 2.

If $v$ has at least three support neighbors of degree 2 in $T$, say $\{u_1, u_2, \ldots, u_l\}$ and $l\geqslant 3$, and set $\overline{T}$ as the tree from $T$ by deleting $\{u_3, u_4, \ldots, u_l\}$ and their respective children, then similar to the discuss as above, $\overline{T}\in \mathcal{T}_t$ and $T$ is obtained from $\overline{T}$ by applying a series of Operation $\mathcal{O}_2$, so $T\in\mathcal{T}_{t}$. Hence we may assume that every vertex has at most two support neighbors of degree 2, denoted by Assumption 3.

Let $F_{t}$ be a $\gamma'_{t}(T)$-set containing non-leaf edges,  $P=v_{0}v_{1}\ldots v_{t}$ the longest path of $T$, say the edge $e_{i}=v_{i}v_{i+1}$. Obviously, $v_{1}$ is a support vertex of degree 2 and each child of $v_{2}$ is a support vertex of degree 2. We root $T$ at the vertex $v_t$.

Since $e_0$ is a leaf edge, $F_t$ must contain $e_1$. Combined with Lemma \ref{Le:1=1NonTriSta} and the choice of $F_t$, it is impossible to contain both $e_2$ and $e_3$ in $F_t$, i.e., $e_2\in F_t$ and  $e_3\notin F_t$ or $e_2\notin F_t$ and $e_3\in F_t$ or $e_2\notin F_t$ and $e_3\notin F_t$.

Combined with Assumptions 2 and 3, $d(v_2)=2$ or $3$. Next, we divide two cases according to the degree of $v_2$.

\noindent\textbf{Case 1.} $d(v_2)=3$.

By Assumptions 1 and 2, $v_2$ has another support child $v'_1$ of degree 2, say $v'_0$ is the child of $v'_1$.

\noindent\textbf{Subcase 1.1.} $e_2\in F_t$.

In this subcase, we can let $\overline{T}=T-\{v_{0},v_{1}\}$. Combined with Lemma \ref{Le:1=1NonTriSta} and the choice of $F_t$, the restriction of $F_t$ on $\overline{T}$ is a TED-set of $\overline{T}$, further,
$\gamma'_{t}(\overline{T})\leqslant \gamma'_{t}(T)-1$.
Combined with an obvious
inequality: $\gamma'(T)\leqslant \gamma'(\overline{T})+1$, we have $\gamma'(\overline{T})+1 \leqslant \gamma'_{t}(\overline{T})+1 \leqslant \gamma'_{t}(T)=\gamma'(T)\leqslant \gamma'(\overline{T})+1$, and so $\gamma'_{t}(\overline{T})$ =$ \gamma'(\overline{T})$. By the inductive hypothesis, there is an edge label of $\overline{T}$ such that $\overline{T}\in \mathcal{T}_{t}$. By Observation \ref{ob2} (\ref{Ob:LEisL1orL2}), $v'_{1}v_2$ is an $S$-edge in $\overline{T}$, combined with Observation \ref{ob2} (\ref{Ob:SedgeisTED}), there are at least two $S$-edges incident with $v_2$ in $\overline{T}$ in either case,  $v_2\in A_2$. We can obtain the tree $T$ by applying Operation $\mathcal{O}_{2}$ from $\overline{T}$ and a new edge $v_{0}v_{1}$, so ${T}\in \mathcal{T}_{t}$.

\noindent\textbf{Subcase 1.2.} $e_2\notin F_t$.

Let $\overline{T}=T-\{v_{0},v_{1},v'_{1},v'_{0}, v_2\}$. Since edges $v'_{1}v'_{0}$ and $e_0$ are leaf edges, combined with Lemma \ref{Le:1=1NonTriSta} and the choice of $F_t$,
the restriction of $F_t$ on $\overline{T}$ is a TED-set of $\overline{T}$, further,
$\gamma'_{t}(\overline{T})\leqslant \gamma'_{t}(T)-2$.

Combined with an obvious inequality:  $\gamma'(T) \leqslant \gamma'(\overline{T})+2$, we have $\gamma'(\overline{T}) +2\leqslant \gamma'_{t}(\overline{T})+2\leqslant \gamma'_{t}(T)=\gamma'(T)\leqslant  \gamma'(\overline{T})+2$, and so $\gamma'_{t}(\overline{T})= \gamma'(\overline{T})$.
By the inductive hypothesis, there is an edge label of $\overline{T}$ such that $\overline{T}\in \mathcal{T}_{t}$. We can obtain the tree $T$ by applying Operation $\mathcal{O}_{5}$ from $\overline{T}$ and a path $v_0v_1v_2v'_1v'_{0}$, so ${T}\in \mathcal{T}_{t}$.

\noindent\textbf{Case 2.}  $d(v_{2})=2$.

Since $d(v_2)=2$, we have $\{e_1,e_2 \}\subseteq F_t$ by the choice of $F_t$. So $(E(v_3)- e_2)\cap F_t=\emptyset$ by Lemma \ref{Le:1=1NonTriSta}.

\begin{claim}\label{1=1v3child}
Let $v'_2$ be a child of $v_3$ other than $v_2$. Then $v'_2$ is a leaf vertex.
\end{claim}

By contradiction. $v'_2$ has at most one support child by symmetry and Assumption 3. Then $v'_2v_3$ belongs to $F_t$ by the choice of $F_t$, a contradiction with Lemma \ref{Le:1=1NonTriSta}.

\begin{claim}\label{Cl:1=1v'4child}
Let $v'_4$ be any non-leaf child of $v_5$. Then each subtree $T_{4'}$ of $T- v'_4$ not containing $v_5$ is isomorphic to one of the graphs in the following figure.
\end{claim}

\begin{figure}[H]
  \centering
  \includegraphics[scale=0.55]{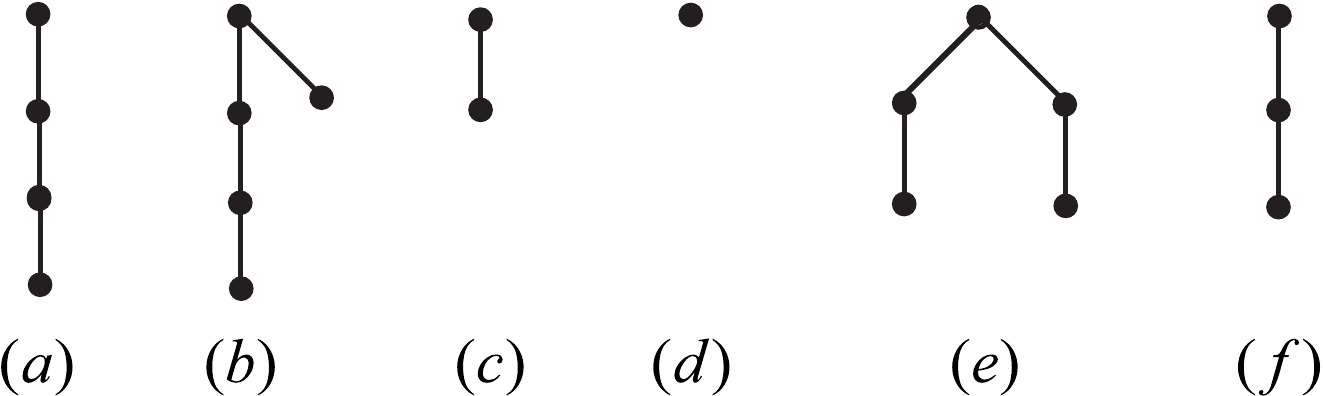}
    \vspace{0.3cm}
  \caption{The subgraphs following $v'_4$.}
  \label{fig:1=1v4Chi}
\end{figure}

 Let $v'_3$ be the child of $v'_4$ in $T_{4'}$. If the length of a longest path starting at $v'_3$ in $T_{4'}$ is 3, combined with symmetry, Claim \ref{1=1v3child} and Assumptions 1, 2, 3, then $T_{4'}$ is isomorphic to (a) or (b). If the length of a longest path starting at $v'_3$ in $T_{4'}$ is 2, combined with Assumptions 1, 2, 3, then $T_{4'}$ is isomorphic to (e) or (f). If the length of a longest path starting at $v'_3$ in $T_{4'}$ is 1, by Assumption 1, then $T_{4'}$ is isomorphic to (c). If the length of a longest path starting at $v'_3$ in $T_{4'}$ is 0, then $T_{4'}$ is (d). Therefore, $T_{4'}$ is isomorphic to one of the graphs in the Fig. \ref{fig:1=1v4Chi}.

 If $v_4$ has one non-leaf child, say $v''_3$, such that the subtree $T_4$ of $T-v_4$ containing $v''_3$ is isomorphic to (e) in Fig. \ref{fig:1=1v4Chi}, then $T_4$ is a $P_5$.
Let $\overline{T}$ be the subtree of $T-T_4$. If $v_4v''_3 \in F_t$, then $F_t-v_4v''_3+e_3-e_2$ is still an ED-set of $T$ of size $|F_t|-1$, a contradiction. Hence $v_4v''_3\notin F_t$ and the restriction of $F_t$ on $\overline{T}$ is a TED-set of $\overline{T}$, further $\gamma'_t(\overline{T})\leqslant \gamma'_t(T)-2$. Combined with an obvious inequality: $\gamma'(T)\leqslant \gamma'(\overline{T})+2$, we have $\gamma'(\overline{T})+2\leqslant \gamma'_t(\overline{T})+2 \leqslant \gamma'_t(T)=\gamma'(T)\leqslant \gamma'(\overline{T})+2$, so $\gamma'(\overline{T})=\gamma'_t(\overline{T})$. By the inductive hypothesis, $\overline{T} \in \mathcal{T}_t$ with an edge labelling. Thus $T$ is obtained from $\overline{T}$ by
applying Operation $\mathcal {O}_5$. So $T \in \mathcal{T}_t$.
In what follows assume that there is no subtree of $T-v_4$ not containing  $v_5$ isomorphic to (e) in Fig. \ref{fig:1=1v4Chi}, denoted by Assumption 4.

\begin{claim}\label{Cl:1=1v5}
If $|E(v_5)\cap F_t|\geqslant 1$ and there is a child, say $v''_4$, of $v_5$ such that there is a subtree $T_{4''}$ of $T-v''_4$ not containing $v_5$ isomorphic to (e) in Fig. \ref{fig:1=1v4Chi}, then $T$ is obtained from $\overline{T}=T- T_{4''}$ by applying Operation $\mathcal {O}_5$.
\end{claim}

Let $v'''_3$ be the child of $v''_4$ in $T_{4''}$. Obviously $T_{4''}$ is a $P_5$. In one case $|E(v_5)\cap F_t|\geqslant 2$, if $v''_4v'''_3\in F_t$, then $F'_t=F_t-v''_4v'''_3+v''_4v_5$ is still a TED-set of $T$. The restriction of $F'_t$ on $\overline{T}$ is a TED-set of $\overline{T}$, further $\gamma'_t(\overline{T})\leqslant \gamma'_t(T)-2$. Similar to the discuss as above, $\overline{T} \in \mathcal{T}_t$. Thus $T$ is obtained from $\overline{T}$ by applying Operation $\mathcal {O}_5$.  If $v''_4v'''_3\notin F_t$, then the restriction of $F_t$ on $\overline{T}$ is a TED-set of $\overline{T}$. Similar to the discuss as above, $T$ is obtained from $\overline{T}$ by applying Operation $\mathcal {O}_5$.

In the other case $|E(v_5)\cap F_t|=1$, say $e'_5=E(v_5)\cap F_t$. Let $x$ be any non-leaf neighbor of $v_5$ in $T$. We claim that $|E(x)\cap F_t|\neq 1$. Indeed, if this is not the case, say $e_x=E(x)\cap F_t$, then $F_t-e'_5-e_x+xv_5$ is an ED-set of $T$, a contradiction. So $v''_4v'''_3\notin F_t$. Combined Lemma \ref{Le:1=1NonTriSta}, similar to the discuss as above, $T$ is obtained from $\overline{T}$ by applying Operation $\mathcal {O}_5$.

In what follows assume that, if $|E(v_5)\cap F_t|\geqslant 1$ and let $v'_4$ be a child of $v_5$, there is no subtree of $T-v'_4$ not containing $v_5$ isomorphic to (e) in Fig. \ref{fig:1=1v4Chi}, denoted by Assumption 5.

By Claim \ref{Cl:1=1v'4child}, let $\{v_3^1, v_3^2, \ldots, v_3^w\}$ be the set of children of $v_4$ such that the subtree of $T-v_4$ containing $v_3^i$ is isomorphic to (a) for $0\leqslant i\leqslant w$,
$\{u_{3}^{1},u_{3}^{2},\ldots,u_{3}^{z}\}$ the set of children of $v_4$ such that the subtree of $T-v_4$ containing $u_{3}^{j}$ is isomorphic to (b) in Fig. \ref{fig:1=1v4Chi} for $0\leqslant j\leqslant z$. Combined with the structure of (a) and (b) and Lemma \ref{Le:1=1NonTriSta}, we have $|(E(v_3^i)-v_3^iv_4)\cap F_t|=1$ and $|(E(u_3^j)-u_3^jv_4)\cap F_t|=1$,
 say $e_v^i=(E(v_3^i)-v_3^iv_4)\cap F_t$ and $e_u^j=(E(u_3^j)-u_3^jv_4)\cap F_t$ for each $i$ and $j$. Then,
\begin{claim}\label{Cl:1=1d(v4)=2}
$w\leqslant 1$. Further, if $w= 1$, then $d(v_4)=2$.
\end{claim}

By contradiction. If $w\geqslant 2$, then $F_t-e_v^1+e_3-e_v^2$ is an ED-set of $T$ of size $|F_t|-1$, a contradiction. So $w\leqslant 1$.

Assume that $d(v_4)\geqslant 3$ when $w=1$. If $z\neq 0$, then $F_t-e_u^1+v_4u_3^1-e_v^1$ is an ED-set of $T$ of size $|F_t|-1$, a contradiction. If there is a subtree of $T-v_4$ not containing $v_5$ isomorphic to one of (c), (d) and (f) in Fig. \ref{fig:1=1v4Chi}, then $E(v_4)\cap F_t\neq \emptyset$ by the choice of $F_t$ and Lemma \ref{Le:1=1NonTriSta}, thus $F_t-e_v^1$ is an ED-set of $T$ of size $|F_t|-1$, a contradiction. Therefore, if $w=1$, then $d(v_4)=2$.

By Claim \ref{Cl:1=1d(v4)=2}, we have the following two claims.
\begin{claim}\label{1=1v4hasnoChid(g)}
There is no subtree of $T-v_4 $ not containing $v_5$ isomorphic to (f) in Fig. \ref{fig:1=1v4Chi}.
\end{claim}

By Claim \ref{Cl:1=1d(v4)=2}, we just need to consider the case $w=0$. By contradiction. If there is a subtree of $T-v_4$ containing one child, say $v'_3$, of $v_4$ isomorphic to (f), then $|E(v'_3)\cap F_t|\geqslant 2$ and  $v_4v'_3\in F_t$. Thus $F_t-v_4v'_3-e_2+e_3$ is an ED-set of $T$ of size $|F_t|-1$ by Lemma \ref{Le:1=1NonTriSta}, a contradiction.

Combined with Assumption 4 and Claim \ref{1=1v4hasnoChid(g)}, there is no  subtree of $T-v_4$ not containing $v_5$ isomorphic to (e) or (f).

\begin{claim}\label{Cl:1=1E(v4)0or2}
If $w=1$, then $E(v_4)\cap F_t=\emptyset$. Otherwise, $|E(v_4)\cap F_t|\neq 1$.
\end{claim}

By contradiction. Assume $E(v_4)\cap F_t\neq\emptyset$ when $w=1$, then $F_t-e_2$ is an ED-set of $T$, a contradiction.
 Assume $|E(v_4)\cap F_t|=1$ when $w=0$, then $e_4\in F_t$ by Claims \ref{Cl:1=1v'4child} and \ref{1=1v4hasnoChid(g)}. Thus $F_t -e_4+e_3-e_2$ is an ED-set of $T$, a contradiction.
Therefore, if $w=1$, then $E(v_4)\cap F_t=\emptyset$. Otherwise, $|E(v_4)\cap F_t|\neq 1$.

Combined with Claims \ref{Cl:1=1v'4child}, \ref{1=1v4hasnoChid(g)} and \ref{Cl:1=1E(v4)0or2}, if $|E(v_4)\cap F_t|\geqslant 2$, then there is a subtree of $T-v_4$ not containing $v_5$ isomorphic to graph (c) in Fig. \ref{fig:1=1v4Chi}, i.e., a $P_2=uv$, say $u$ is a child of $v_4$.
By Claim \ref{Cl:1=1d(v4)=2}, the subtree $T^a$ of $T-v_4$ containing $v_3$ is isomorphic to (b). Let $\overline{T}$ be the subtree of $T-T^a$. By $|E(v_4)\cap F_t|\geqslant 2$, the restriction of $F_t$ on $\overline{T}$ is a TED-set of $\overline{T}$, further $\gamma'_t(\overline{T})\leqslant \gamma'_t(T)-2$. Combined with an obvious inequality: $\gamma'(T)\leqslant \gamma'(\overline{T})+2$, we have $\gamma'(\overline{T})+2\leqslant \gamma'_t(\overline{T})+2 \leqslant \gamma'_t(T)=\gamma'(T)\leqslant \gamma'(\overline{T})+2$, and so $\gamma'(\overline{T})=\gamma'_t(\overline{T})$. By the inductive hypothesis, $\overline{T} \in \mathcal{T}_t$ with an edge labelling.
In $\overline{T}$, by Observation \ref{ob2} (\ref{Ob:LEisL1orL2}), the leaf edge $uv$ is an $L_1$-edge and $v_4u$ is an $S$-edge. Combined with Observation \ref{ob2}  (\ref{Ob:SedgeisTED}), $v_4\in A_2$.  Therefore $T$ is obtained from $\overline{T}$ by applying Operation $\mathcal{O}_3$. So $T \in \mathcal{T}_t$.

If $E(v_4)\cap F_t=\emptyset$, then there is no subtree of $T-v_4$ not containing $v_5$ isomorphic to (c) or (d). Combined with Claims \ref{Cl:1=1v'4child}, \ref{Cl:1=1d(v4)=2}, Assumption 4 and the above analysis, then we may assume that each subtree of $T-v_4$ not containing $v_5$ is isomorphic to (a) or (b), denoted by Assumption 6.

By Assumption 6, we can divide two subcases to discuss according to the subtree $T_4$ of $T-v_4$ containing $v_3$ is isomorphic to (a) or (b) as follow.

\noindent\textbf{Subcase 2.1.} $T_4$ is isomorphic to (a).

By Claims \ref{Cl:1=1d(v4)=2} and \ref{Cl:1=1E(v4)0or2}, $d(v_3)=d(v_4)=2$ and $E(v_4)\cap F_t=\emptyset$. Obviously, $e_4\notin F_t$ and $|E(v_5)\cap F_t|\geqslant 1$.

\begin{claim}\label{Cl:1=1E(v5)isBig2}
$|E(v_5)\cap F_t|\geqslant 2$.
\end{claim}

If $|(E(v_5)- e_4)\cap F_t|=1$, say $e'_5=E(v_5)\cap F_t$, then $F_t-e'_5 + e_4- e_2$ is an ED-set of $T$ of size $|F_t|-1$, a contradiction.

Combined with $|E(v_5)\cap F_t|\geqslant 2$ and $E(v_4)\cap F_t=\emptyset$, there is a child other than $v_4$, say $v'_4$, of $v_5$ such that $v_5v'_4\in F_t$. Obviously, $v'_4$ is not a leaf. Then

\begin{claim}\label{Cl:1=1v5hasSupChid}
$v'_4$ is a support vertex of degree 2.
\end{claim}

We first show that $v'_4$ is a support vertex. Assume to the contrary that $v'_4$ has no leaf children, combined with Claim \ref{Cl:1=1v'4child}, Assumption 5 and Lemma \ref{Le:1=1NonTriSta}, every subtree of $T-v'_4$ not containing $v_5$ is isomorphic to (a) or (b) in Fig. \ref{fig:1=1v4Chi}. Obviously  $F_t-v_5v'_4$ is still an ED-set of $T$ of size $|F_t|-1$, a contradiction. Therefore, $v'_4$ is a support vertex.

If $d(v'_4)\geqslant 3$, then we denote by $T'$ a subtree of $T-v'_4$ containing a non-leaf child of $v'_4$. Combined with Assumption 5 and  Claim \ref{Cl:1=1v'4child}, $T'$ is not isomorphic to (a) or (e) in Fig. \ref{fig:1=1v4Chi}. By Lemma \ref{Le:1=1NonTriSta}, $T'$ is not isomorphic to (c) or (f). Hence $T'$ is isomorphic to (b), say $v'_3$ the non-leaf child of $v'_4$, $v'_2$ the non-leaf child of $v'_3$. Obviously, $F_t-v_5v'_4-v'_2v'_3+v'_3v'_4$ is an ED-set of $T$ of size $|F_t|-1$, a contradiction. So $d(v'_4)=2$.

Since $d(v_2)=d(v_3)=d(v_4)=2$, the subgraph induced by $\{v_0, v_1, v_2, v_3 \}$ is $P_4$. Let $\overline{T}= T-\{v_0,v_1,v_2,v_3\}$. By $E(v_4)\cap F_t=\emptyset$, the restriction of $F_t$ on $\overline{T}$ is a TED-set of $\overline{T}$, further $\gamma'_t(\overline{T})\leqslant \gamma'_t(T)-2$. Combined with an obvious inequality: $\gamma'(T)\leqslant \gamma'(\overline{T})+2$, we have $\gamma'(\overline{T})+2\leqslant \gamma'_t(\overline{T})+2 \leqslant \gamma'_t(T)=\gamma'(T)\leqslant \gamma'(\overline{T})+2$, and so $\gamma'(\overline{T})=\gamma'_t(\overline{T})$. By the inductive hypothesis, $\overline{T} \in \mathcal{T}_t$ with an edge labelling. By Observation \ref{ob2} (\ref{Ob:LEisL1orL2}), $e_4$ is an $L_2$- or $L_1$-edge. Combined with Claim \ref{Cl:1=1v5hasSupChid} and Observation \ref{ob2} (\ref{Ob:EndpOfL1}), (\ref{Ob:SedgeisTED}), $e_4$ is an $L_2$-edge. Since $d_{\overline{T}}(v_4)=1$, $v_4\in B$.  Therefore $T$ is obtained from $\overline{T}$ by applying Operation $\mathcal{O}_4$, $T\in \mathcal{T}_t$.

\noindent\textbf{Subcase 2.2.} $T_4$ is isomorphic to (b).

In this subcase, let $v'_4$ be any non-leaf child of $v_5$, by symmetry and Assumption 6, there is no subtree of $T-v'_4$ not containing $v_5$ isomorphic to (a) in Fig. \ref{fig:1=1v4Chi}, and each subtree of $T-v_4$ not containing $v_5$ is isomorphic to (b), 
i.e., a $P_5$. Let $\overline{T}=T-T_4$. Since $e_3\notin F_t$ by Lemma \ref{Le:1=1NonTriSta}, the restriction of $F_t$ on $\overline{T}$ is a TED-set of $\overline{T}$, further $\gamma'_t(\overline{T})\leqslant \gamma'_t(T)-2$. Combined with an obvious inequality: $\gamma'(T)\leqslant \gamma'(\overline{T})+2$, we have $\gamma'(\overline{T})+2\leqslant \gamma'_t(\overline{T})+2 \leqslant \gamma'_t(T)=\gamma'(T)\leqslant \gamma'(\overline{T})+2$, further $\gamma'(\overline{T})=\gamma'_t(\overline{T})$.
By the inductive hypothesis, $\overline{T} \in \mathcal{T}_t$ with an edge labelling. Combined with the structure of (b) and Observation \ref{ob2} (\ref{Ob:L2NbSedge}), then in $\overline{T}$, all edges connecting $v_4$ and its children being $L_1$-edges, so $v_4\in C$ by Observation \ref{ob2} (\ref{Ob:EndpOfL1}). If in $\overline{T}$, $e_4$ is an $L_2$-edge or $e_4$ is an $L_1$-edge and adjacent to a leaf edge, then $T$ is obtained from $\overline{T}$ by applying Operation $\mathcal{O}_3$. In what follows we assume that in $\overline{T}$, $e_4$ is an $L_1$-edge and adjacent to non-leaf edges, denoted by Assumption 7.
Note that in $\overline{T}$, there is only one $S$-edge in $E_{\overline{T}}(v_5)$, say $e'_5$, and $v_5\in A_1$ in $\overline{T}$.

By Lemma \ref{Le:TIs1=1}, all $S$-edges in $\overline{T}$ construct a minimum total edge dominating set $D(\overline{T})$.
Then, $F'_t=D(\overline{T})+e_1+e_2$ is a minimum total edge dominating set of $T$ by $\gamma'_t(\overline{T})+2=\gamma'_t(T)$. Further, $|E(v_5)\cap F'_t|=1 $, Claim \ref{Cl:1=1v5} still holds. Then, we have the following claim:

\begin{claim}\label{1=1e5L1v6L2}
$e_5$ is an $L_1$-edge in $E_{\overline{T}}(v_5)$, and  all edges in $(E_{\overline{T}}(v_6)- e_5)$ are $L_2$-edges.
\end{claim}

Let $v'_4$ be any non-leaf child of $v_5$ other than $v_4$ and $v'_3$ any child of $v'_4$ in $\overline{T}$. There is no subtree of $\overline{T}- v'_4$ containing $v'_3$ isomorphic to (e) in $\overline{T}$  by Assumption 5 and Claim \ref{Cl:1=1v5}. We claim that the length of a longest path starting at $v'_3$ in the subtree $T'$ of $T-v'_4 $ containing $v'_3$ is 3 or 1. If the length of a longest path starting at $v'_3$ in $T'$ is 2 or 0, by Lemma \ref{Le:1=1NonTriSta} and Observation \ref{ob2} (\ref{Ob:LEisL1orL2}), (\ref{Ob:SedgeisTED}), then $v'_4\in A_1$ in $\overline{T}$, a contradiction by Observation \ref{ob2} (\ref{Ob:EndpOfL1}) and Lemma \ref{Le:1=1NonTriSta}. By symmetry, Claim \ref{Cl:1=1v'4child} and Assumptions 1, 2, 3, we know that $T'$ is isomorphic to (b) or (c).
Let $\{z^1_4,\ldots, z^h_4 \}$ be the set of children of $v_5$ such that there is a subtree $T^r_4$ of $T- z^r_4 $ not containing $v_5$ is isomorphic to (b) in $\overline{T}$ for $1\leqslant r\leqslant h$. Let $z^r_3$ be the child of $z^r_4$ in $T^r_4$. By the structure of (b), then $|E(z^r_3)\cap F'_t|=1$, say $e^r=(E(z^r_3)- z^r_3z^r_4)\cap F'_t$ for each $r$. If $E(v_6)$ has an $S$-edge other than $e_5$, then $F'_t-\{e^1,\ldots,e^h \} +\{z^1_3z^1_4,\ldots,z^h_3z^h_4 \}-e'_5$ is an ED-set of $T$ of size $|F'_t|-1$, a contradiction. So all edges in $(E(v_6)- e_5)$ are $L_1$- or $L_2$-edges. By Observation \ref{ob2} (\ref{Ob:L2NbSedge}), (\ref{Ob:SedgeisTED}), $e_5$ is an $L_1$-edge.

By contradiction. If there is one edge $e'_6=v_6v'_7$ in $(E_{\overline{T}}(v_6)- e_5)$ is an $L_1$-edge, then there is exactly one $S$-edge $e''_6$ incident with $v'_7$ by Observation \ref{ob2} (\ref{Ob:EndpOfL1}). Thus $F'_t-e''_6+e'_6-\{e^1,\ldots,e^h \} +\{z^1_3z^1_4,\ldots,z^h_3z^h_4 \}-e'_5$ is an ED-set of $T$ of size $|F'_t|-1$, a contradiction. Therefore, all edges in $(E_{\overline{T}}(v_6)- e_5)$ are $L_2$-edges.

Combined  with the structure of (b) and Claim \ref{1=1e5L1v6L2}, all $L_1$-edges in $(E_{\overline{T}}(v_4)-e_4)$ are adjacent to a leaf edge, and there exist a $P_4$ starting at $v_4$, whose edges are labelled as $L_1$, $L_1$, $L_2$ consecutively, and all edges in $(E_{\overline{T}}(v_6)-e_5)$ are $L_2$-edges. Hence we can obtain $T$ from $\overline{T}$ by applying Operation $\mathcal{O}_3$.
\end{proof}

As an immediate consequence of Lemmas \ref{Le:TIs1=1} and  \ref{Le:1=1IsT}, we have the following characterization of $(\gamma'_{t}=\gamma')$-trees.
 \begin{theorem}
 A tree is a $(\gamma'_{t}=\gamma')$-tree if and only if $T\in \mathcal{T}_{t}$.
 \end{theorem}

\section{Acknowledgements}

This work was funded in part by National Natural Science Foundation
of China (Grants No. 11571155, 11201205).


\begin{thebibliography}{}

\bibitem{hhs98}T.W. Haynes, S.T. Hedetniemi, P.J. Slater, Fundermentals of Domination in Graphs, Marcel Dekker, New York, 1998.

\bibitem{hk93}J.D. Horton, K. Kilakos, Minimum edge dominating sets, SIAM J. Discrete Math. 6(3) (1993) 375-387.

\bibitem{k72}R. Karp, Reducibility among combinatorial problems, Complexity of Computer Computations, R.E. Miller and J.W. Thatcher, eds., Plenum Press, New York, 1972, pp. 85-104.

\bibitem{k98}K. Kilakos, On the complexity of edge domination, Master's Thesis, University of New Brunswick, New Brunswick, Canada, 1998.

\bibitem{kp91}V.R. Kulli, D.K. Patwari, On the edge domination number of a graph, in: Proceedings of the Symposium on Graph Theory and Combinatorics, Cochin, 1991, in: Publication, vol. 21,  Centre Math. Sci. Trivandrum, 1991, pp. 75-81.

\bibitem{l68} C.L. Lru, Introduction to Combinatorial Mathematics, McGraw-Hill, New York, 1968.

\bibitem{mh77} S. Mitchell, S.T. Hedetniemi, Edge domination in trees, Congr. Numer. 19 (1977) 489-509.

\bibitem{ms13} M.H. Muddebihal, A.R. Sedamkar, Characterization of trees with equal edge domination and end edge domination numbers, Mathematical Theory and Modeling, 5 (2013) 33–42.

\bibitem{pc16} M.N.S. Paspasan, S.R. Canoy,  Edge domination and total edge domination in the join of graphs, Appl. Math. Sci. 10 (2016) 1077-1086.


\bibitem{v14}S. Velammal, Equality of connected edge domination and total edge domaination in graphs, International Journal of Enhanced Research in Science Technology and Engineering 5 (2014) 198-201.

\bibitem{x06} B. Xu, Two classes of edge domination in graphs, Discrete Appl. Math. 154 (2006) 1541-1546.

\bibitem{y81} M. Yannakakis, Edge-deletion problems, SIAM J. Comput. 10 (1981) 297-309.

\bibitem{yg80} M. Yannakakis, F. Gavril, Edge dominating sets in graphs, SIAM J. Appl. Math. 38 (1980) 364-372.

\bibitem{zlm14} Y.C. Zhao, Z.H. Liao, L.Y. Miao, On the algorithmic complexity of edge total domination, Theoret. Comput. Sci. 6 (2014) 28-33.
\end{thebibliography}
\end{document}